\documentclass[a4paper,11pt]{article}
\usepackage{mathrsfs}
\usepackage{amsfonts}
\usepackage[all]{xy}
\usepackage{amsmath}
\usepackage{arcs}
\usepackage{amscd}
\usepackage{extarrows}
\usepackage{CJK}
\usepackage[pdftex]{graphicx,xcolor}
\usepackage{bm}
\usepackage{amssymb}
\usepackage{amscd,latexsym,amsthm,amsfonts,amssymb,amsmath,amsxtra}
\newtheorem{prop}{Proposition}
\newtheorem{thm}{Theorem}
\newtheorem{lemma}{Lemma}
\newtheorem{cor}{Corollary}
\newtheorem{rem}{Remark}

\author{Zhangjie Wang}

\newcommand{\En}{E^{(n)}}

\newcommand{\Leg}[2]{{\brlr{\frac{#1}{#2}}}}
\newcommand{\ALeg}[2]{{\Sqlr{\frac{#1}{#2}}}}
\newcommand{\aLeg}[2]{{\SqBig{\frac{#1}{#2}}}}

\newcommand{\arr}{\ar@}
\newcommand{\arro}{\ar@/}

\newcommand{\Kc}{{K^\times}}

\newcommand{\uset}{\underset}
\newcommand{\oset}{\overset}

\newcommand{\Ma}[4]{{
\left(
  \begin{array}{cc}
    #1 & #2 \\
    #3 & #4 \\
  \end{array}
\right)
}}

\newcommand{\Ba}[2]{{\langle #1, #2\rangle}}

\newcommand{\Babig}[2]{{\big\langle #1, #2\big\rangle}}

\newcommand{\brlr}[1]{{\left( #1\right)}}
\newcommand{\brbig}[1]{{\big(#1 \big)}}
\newcommand{\brBig}[1]{{\Big(#1 \Big)}}

\newcommand{\Brlr}[1]{{\left\{ #1\right\}}}
\newcommand{\Brbig}[1]{{\big\{ #1\big\}}}
\newcommand{\BrBig}[1]{{\Big\{ #1\Big\}}}

\newcommand{\Sqlr}[1]{{\left[#1 \right]}}

\newcommand{\SqBig}[1]{{\Big[#1 \Big]}}

\font\cyr=wncyr10 scaled \magstep 1 
\newcommand{\Sha}{\mbox{\cyr X}} 

\newcommand{\diag}{{\mathrm{diag}}}

 \renewcommand{\Im}{{\mathrm{Im}}}

\newcommand{\Ker}{{\mathrm{Ker}}}

 \newcommand{\rank}{{\mathrm{rank}}}

 \newcommand{\Sel}{{\mathrm{Sel}}}

\renewcommand{\mod}{\mathrm{mod}}

 \newcommand{\hp}{{\hat {p}}}

\newcommand{\tda}{{\tilde {a}}} \newcommand{\tdb}{{\tilde {b}}}
\newcommand{\tdc}{{\tilde {c}}} \newcommand{\tdd}{{\tilde {d}}}

\newcommand{\bara}{{\bar {a}}} \newcommand{\barb}{{\bar {b}}}
\newcommand{\barc}{{\bar {c}}}

\newcommand{\bA}{{\mathbb {A}}} 
 
 \newcommand{\bF}{{\mathbb {F}}}

 \newcommand{\bP}{{\mathbb {P}}}
\newcommand{\bQ}{{\mathbb {Q}}} \newcommand{\bR}{{\mathbb {R}}}

 \newcommand{\bZ}{{\mathbb {Z}}}

\newcommand{\cA}{{\mathcal {A}}} 
 \newcommand{\cD}{{\mathcal {D}}}

\newcommand{\cO}{{\mathcal {O}}}

\newcommand{\fa}{{\mathfrak {a}}} 
\newcommand{\fc}{{\mathfrak {c}}} \newcommand{\fd}{{\mathfrak {d}}}

\newcommand{\fm}{{\mathfrak {m}}} 
 \newcommand{\fp}{{\mathfrak {p}}}

\renewcommand{\Sel}{{{\mathrm{Sel}}}}

\newcommand{\ideal}[1]{ {\brlr{#1, \alpha_0} }}
\newcommand{\idealc}[1]{ {\Sqlr{\ideal{#1}} }}

\numberwithin{equation}{section}
\title{Congruent Elliptic Curves with Non-trivial Shafarevich-Tate Groups}

\begin{document}
\maketitle

\begin{abstract}

We study a subclass of congruent elliptic curves $\En: y^2=x^3-n^2x$, where  $n$ is a positive integer congruent to $1\pmod 8$ with all prime factors congruent to $1\pmod 4$.   We characterize such $\En$ with Mordell-Weil rank zero and $2$-primary part of Shafarevich-Tate group  isomorphic to $\brbig{\bZ/2\bZ}^2$.
We also discuss   such  $\En$  with 2-primary part of Shafarevich-Tate group  isomorphic to $\brbig{\bZ/2\bZ}^{2k}$ with $k\ge2$.\footnote{
Congruent Elliptic Curves with Non-trivial Shafarevich-Tate Groups.SCIENCE CHINA Mathematics , Doi: 10.1007/s11425-015--0741-3 }

\end{abstract}

\begin{CJK}{GBK}{kai}
\maketitle

\section{Introduction}
A positive integer $n$ is called  a congruent number if it is the area of a right triangle with rational lengths, which is equivalent the  congruent elliptic curve
$$\En: y^2=x^3-n^2x $$
has positive Mordell-Weil rank.   Heegner \cite{Heegner1952} proved that   any prime number $p\equiv 5\pmod  8$  is a congruent number. Recently, Tian \cite{Tian-congruent} introduced
the induction method involving $L$-value and as a corollary he proved that:
for any $k\ge 1$, there are infinitely many square-free congruent number
$n\equiv 5 \pmod 8 $ with exactly $k$ prime divisors.

In this paper, we mainly study rank zero congruent elliptic curves with $2$-primary part of Shafarevich-Tate group non-trivial. Li-Tian \cite{LiDeLangTian}  considered those $n$ with  all prime factors   congruent to $ 1\pmod 8$.  They found a sufficient condition for $n$ to be a  non-congruent  number  with  $\Sha(\En/\bQ)[2^\infty]\simeq \brbig{\bZ/2\bZ}^2$. Rhoades \cite{rhoades20092} suspected that their condition is also necessary,  which is  proved in Theorem \ref{lem-LiTian} of this paper.  Under our case, Ouyang-Zhang \cite{ShenxingZhangYiouyang}  also  found a sufficient  condition  for $n\equiv 1  \pmod 8 $ non-congruent with $\Sha(\En/\bQ)[2^\infty]\simeq \brbig{\bZ/2\bZ}^2$.

To state our characterization, we introduce some notations.  For  a positive square-free integer $n$, let $\cA_n$ denote the   ideal class group of $K=\bQ(\sqrt{-n})$ and $D$ be  the fundamental discriminant of $K$. We denote $\cD$ to be the set of all positive square-free divisors of $D$. Then $\cD$ has a group structure.  Gauss genus theory(  see Proposition \ref{gauss} ) implies that there is a  $2$ to $1$ correspondence $\theta$ between $\cD\cap N(\Kc)$ and  $\cA_n[2]\cap2\cA_n$. Here $N(\Kc)$ denotes the image of the norm map from $\Kc$ to $\bQ^\times$,  $\cA_n[2]$ denotes the $2$-elementary subgroup of $\cA_n$, and $2\cA_n$ denotes the subgroup  of  the square elements in $\cA_n$.  We define the $2^i$-rank $h_{2^i}(n)$ of $\bQ(\sqrt{-n})$  by $\rank_{\bF_2} 2^{i-1}\cA_n/2^i\cA_n$ for $i\ge 1$. Then we have $h_4(n)=\rank_{\bF_2} \cA_n[2]\cap 2\cA_n$. If $h_4(n)=1$, then there exist $d_1,d_2\in \cD$ corresponding  to the non-trivial element in $\cA_n[2]\cap 2\cA_n$. Moreover if $n\equiv1\pmod8$ with all prime factors congruent to $1\pmod4$, then the product of odd part $d_i'$ of $d_i$ equals $n$. This is because  the kernel of $\theta$ is $\Brbig{1,n}$ in this case. Then we denote $d=d(n)$ be the maximal integer of $d_1', d_2'$.

Now we can state our first main result.
\begin{thm}\label{mainthm1}
Let $n$ be a positive square-free integer congruent to  $1 \pmod 8 $ with all prime factors congruent to $1\pmod 4 $ distinct and $h_4(n)=1$, then the following are equivalent:
\begin{enumerate}
\item[(i)] ${h_8(n)\equiv \frac{d(n)-1}4 \pmod 2 }$;
\item[(ii)] $\rank_\bZ \En(\bQ)=0$ and $\Sha(\En/\bQ)[2^\infty]\simeq \brlr{\bZ/2\bZ}^2$.
\end{enumerate}
\end{thm}

The  strategy of the proof of Theorem \ref{mainthm1} is the following.
\begin{itemize}
\item Note that a necessary condition of (ii) holds is $s_2(n)=2$, where $s_2(n)$ is the pure $2$-Selmer rank given by $\rank_{\bF_2}\Sel_2(\En)/\En[2](\bQ)$.  By
 Monsky matrix (\S 2.1) and Redei matrix (\S 3.1), we deduce that
 $s_2(n)=2$ is equivalent to  $h_4(n)=1$.
\item  From the  short exact sequence $0\to \En[2]\to \En[4]\oset{\times 2}\to \En[2]\to 0$, we derive the following long exact sequence
\begin{eqnarray}\label{longexactseq}
0\to  \En[2](\bQ)/2\En[4](\bQ) \to \Sel_2(\En)  \to \Sel_{4}(\En) \to
\Im \Sel_{4}(\En)\to 0
\end{eqnarray}
Moreover    Cassels \cite{Cassels1998CasPairing} introduced a pairing on $\Sel_2(\En)/\En[2](\bQ)$  with kernel given by
 $\Im\Sel_4(\En)/\En[2](\bQ)$ (\S 2.2). Note that (ii) is equivalent to $\#\Sel_2(\En)=\#\Sel_4(\En)$, which is equivalent to $\#\En[2](\bQ)/2\En[4](\bQ)=\#\Im\Sel_4(\En)$ by the long exact sequence. From the Cassels pairing, we know (ii) is equivalent to  the Cassels pairing is non-degenerate.
\item   Cassels pairing is computable(\S 4.1), its non-degeneration in this case is equivalent to a quartic residue symbol.
\item Gauss genus theory (\S 3.2) reduces this residue symbol to $8$-rank $h_8(n)$ of $\bQ(\sqrt{-n})$.
\end{itemize} \quad

For the case $k\ge2$,  we have the following second main result.
\begin{thm}\label{mainthm}
Let $n$ be a  positive square-free integer with all  prime factors congruent to $1 \pmod 8$ and $k$ be a positive integer.  Assume that either    $h_8(n)=k-1$   with the ideal class $[(2,\sqrt{-n})]\not\in 4\cA_n$ or  $h_8(n)=k$. If $n$ has a decomposition $d_1\cdots d_k$ with $d_i>1$ satisfying the following ideal class group conditions:
\begin{enumerate}
\item $h_4(d_i)=1,  1\le i\le k$;
\item $\Leg{p_{i_1}}{p_{i_2}}=1$ for any prime divisor $p_{i_j}$ of $d_{i_j}$ with $i_1\not=i_2$, where $\Leg{\cdot}{\cdot}$ is the Legendre symbol;
\item $h_8(d_i)=0, 1\le i\le k$.
\end{enumerate}
Then
\begin{equation}\label{eq:anysha-rank0}
\rank_{\bZ} \En(\bQ)=0, \;\;\Sha(\En/\bQ)[2^\infty]\simeq \brbig{\bZ/2\bZ}^{2k}
\end{equation}
\end{thm}


We remark that for $k=2$, the condition is not empty by  Remark \ref{rem:non-empty}.  And in \S4.2  we can get congruent elliptic curves such that (\ref{eq:anysha-rank0}) holds under   a more general condition (see Theorem \ref{mainthm2}).  As to the distribution result    of elliptic curves in Theorem \ref{mainthm1} and \ref{mainthm}, we will study this in a
coming paper \cite{zhangjiewang2015congDist} but with Theorem \ref{mainthm} limited to $k=2$.

\section{Congruent Elliptic Curves }

Let $\En$ be the congruent elliptic curve $y^2=x^3-n^2x$ defined over $\bQ$, where $n$ is a positive square-free integer. Then from  classical result,  we can identify the $2$-Selmer group of $\En$ with:
\[\Sel_2(\En)=\BrBig{\Lambda=(d_1,d_2,d_3)\in \brbig{\bQ^\times/{\bQ^\times2}}^{3} \Big| D_\Lambda(\bA)\not=\emptyset,\; d_1d_2d_3\in \bQ^{\times2} } \]
where $\bA$ is the adele ring over $\bQ$ and  for $\Lambda=(d_1,d_2, d_3), D_\Lambda$ is a genus one curve in $\bP^3$ defined by:
\begin{equation}
\label{eq:Sel2En}
\begin{cases}
H_1:& -nt^2+d_2u_2^2-d_3u_3^2=0\\
H_2:& -nt^2+d_3u_3^2-d_1u_1^2=0\\
H_3:& 2nt^2+d_1u_1^2-d_2u_2^2=0
\end{cases}
\end{equation}
Under this  identification,  $\En(\bQ)/2\En(\bQ)$ is given  by: $(x-n,x+n,x)$ if
$ (x,y)\in \En(\bQ)-\En(\bQ)[2] $ and  $\En(\bQ)[2]$ corresponds to
$$ \BrBig{ (2,2n,n), (-2n,2,-n), (-n,n,-1), (1,1,1) }$$
Then these  $4$ elements  correspond to $(n,0), (-n,0), (0,0), O$ respectively under the Kummer map.

\subsection{Monsky Matrix}

Monsky ( see the appendix of Heath-Brown \cite{HeathBrownMonsky1993Selmergroup} )  represented $\Sel_2(\En)/\En(\bQ)[2]$ with the kernel of a $2k\times 2k$ matrix $M$ over $\bF_2$, which is closely related to $4$-rank of $\bQ(\sqrt{-n})$.

Before   introducing the Monsky matrix, we define  the additive Legendre symbol $\ALeg{a}{p}$, where $p$ is a prime and $a$ an integer  coprime to $p$, it is $0$ if the Legendre symbol $\Leg ap=1$ and $1$ otherwise.

For a positive square-free odd integer $n=p_1\cdots p_k$, the corresponding Monsky matrix $M_n$ is defined by:
\[M_n=\Ma{A+D_{-2}}{D_2}{D_2}{A+D_2}\]
where $D_u=\diag\Brlr{\SqBig{\frac u{p_1}},\cdots,\SqBig{\frac u{p_k}}}$ and  $A=(a_{ij})_{k\times k}$ with $a_{ij}=  \ALeg{p_j}{p_i} $ if $i\not=j$, and $a_{ii}=\sum_{l\not=i}a_{il}$.

To connect $\Sel_2(\En)/\En(\bQ)[2]$ with the kernel of corresponding Monsky matrix, we first choose representatives of $\Sel_2(\En)/\En(\bQ)[2]$  with nice properties, see the following lemma:
\begin{lemma} \label{HeathBrownMonsky1993Selmergroup}
For a positive square-free odd integer $n$, the following holds:
\begin{enumerate}
\item[(1)] Any $\Lambda\in \Sel_2(\En)$ can be represented by $(d_1,d_2,d_3)$ with $d_i$ square-free integers, $d_2>0, d_1|2n, d_2|2n $ and $d_1,d_2$ have the same parity;
\item[(2)] Assume that all elements in $\Sel_2(\En)$ have representatives as in (1), then for any $\Lambda\in \Sel_2(\En)$, there is a unique $\Lambda_0\in\En(\bQ)[2]$
such that $\Lambda\Lambda_0=(d_1,d_2,d_3)$ with $d_1>0, d_2>0, d_1|n, d_2|n$;
\end{enumerate}
\end{lemma}
\begin{proof} (1): Since $d_i\in \bQ^\times/\bQ^{\times2}$, we may chose $d_i$ square-free integers. For $\Lambda\in\Sel_2(\En)$  we have $D_\Lambda(\bR)\not=\emptyset$, this implies that $d_2>0$.

If $d_1\nmid2n$, then there is a prime divisor $p|d_1$ with $p\nmid 2n$. Since $d_1d_2d_3\in \bQ^{\times2}$, hence $p|d_2d_3$, we assume $p|d_2$, which implies $p\nmid d_3$. Consider $D_\Lambda(\bQ_p)$: from $H_3$ we know that $p|t$, then use $H_1$ and $p|d_2, p \nmid  d_3$ we get $p|u_3$. Comparing $p$-divisibility of two sides of $H_2$, we get  $p|u_1$. Again consider $H_3$, then $p|u_2$ by counting $p$-divisibility of two sides. We arrive at $p|(t,u_1,u_2,u_3)$, then we can easily get they are infinitely $p$-divisible, which is impossible, hence $d_1|2n$. Similarly $d_2|2n$.

$d_1, d_2$ has same parity: as else $2$ divide exact one of $d_1,d_2$, for example $2|d_1$, then $2 \nmid  d_2$, this implies $2|d_3$, so from $H_3$ we know $2|u_2$ and from $H_2$ we know $2|t$, then by $2$-divisibility of $H_3$, we get $2|u_1$, again by $2$-divisibility of $H_2$, we get $2|u_3$, hence $2|(t,u_1,u_2,u_3)$, same argument as above implies this is impossible, hence $d_1, d_2$ has same parity.\\

(2): If $d_1$ is even, then by (1) so is $d_2$, then $\Lambda_0=(2,2n,n)$ is the unique element in $\En(\bQ)[2]$ with the required property if $d_1>0$, else $(-2n,2,-n)$ is the unique one. If $d_1$ is odd, then so is $d_2$, if $d_1<0$ then $(-n,n,1)$ is the unique element satisfying the required property, else $(1,1,1)$ is.
\end{proof}

For the representatives of $\Sel_2(\En)/\En(\bQ)[2]$ as in (2) of Lemma \ref{HeathBrownMonsky1993Selmergroup}, we have a bijection
\begin{equation}\label{eq:sel-Monsky}
\Sel_2(\En)/\En(\bQ)[2] \longrightarrow  \BrBig{ x\in \bF_2 ^{2k}\;\Big|\; M_n x=0\;}  \end{equation}
given by  $(d_1,d_2,d_3)  \mapsto \brbig{v_{p_1}(d_1),\cdots,v_{p_k}(d_1), v_{p_1}(d_2),\cdots v_{p_k}(d_2)}^T$, where $v_p(d)$ denotes the $p$-adic valuation of $d$.

Now we explain why (\ref{eq:sel-Monsky}) is a bijection: For $\Lambda$ chosen as in (2) of Lemma \ref{HeathBrownMonsky1993Selmergroup},  we have $D_\Lambda(\bQ_p)\not=\emptyset$ for $p \nmid   n$, while for $p|n$ the local solvability of $D_\Lambda$ is:
\begin{enumerate}
\item[(1)] $p \nmid  d_1, p \nmid  d_2:\;\Leg{d_1}p=\Leg{d_2}p=1$;
\item[(2)] $p \nmid  d_1, p|d_2: \; \Leg{2d_1}p=\Leg{2n/d_2}p=1$;
\item[(3)] $p|d_1,p \nmid  d_2:\; \Leg{-2n/d_1}p=\Leg{2d_2}p=1$;
\item[(4)] $p|d_1, p|d_2:\; \Leg{-n/d_1}p=\Leg{n/d_2}p=1$.
\end{enumerate}

Assume $d_1=\prod_1^kp_i^{x_i}, d_2=\prod_1^k p_i^{y_i}$, then the above says that: for $p_i$
\begin{eqnarray*}
x_i\sum_{l\not=i}\aLeg{p_l}{p_i}+\sum_{j\not=i}x_j\aLeg{p_j}{p_i}&=&
x_i\aLeg{n/d_1}{p_i}+(1-x_i)\aLeg{d_1}{p_i}
=\aLeg{-2}{p_i}x_i+\aLeg{2}{p_i}y_i\\
y_i\sum_{l\not=i}\aLeg{p_l}{p_i}+\sum_{j\not=i}y_j\aLeg{p_j}{p_i}&=&
y_i\aLeg{n/d_2}{p_i}+(1-y_i)\aLeg{d_2}{p_i}
=\aLeg{2}{p_i}x_i+\aLeg{2}{p_i}y_i
\end{eqnarray*}
Hence the map (\ref{eq:sel-Monsky}) is a bijection.

Moreover, Monsky proved that $M$ has even rank if $n\equiv 1, 3\pmod 8$.

\subsection{Cassels Pairing}

Cassels \cite{Cassels1998CasPairing} introduced a skew-symmetric bilinear pairing on the $\bF_2$-vector space\\
 $\Sel_2(\En)/\En(\bQ)[2]$:
\begin{equation}\label{eq:casselspair}
\Ba{ }{ }: \Sel_2(\En)/\En(\bQ)[2]\times \Sel_2(\En)/\En(\bQ)[2] \to \Brlr{ \pm 1 }
\end{equation}
It is defined as follows: for any $\Lambda=(d_1,d_2,d_3), \Lambda'=(d_1',d_2',d_3')\in \Sel_2(\En)$, let $D_\Lambda$ be the  genus one curve corresponding to $\Lambda$ and $Q_i\in H_i(\bQ)$ be a global point on $H_i$, the existence of $Q_i$ follows from Hasse-Minkovski principle, since $H_i$ is locally solvable everywhere as $D_\Lambda$  is so. Let $L_i$ be the tangent line of $H_i$ at $Q_i$ and view it as a hyperplane in $\bP^3$. Let $P=(P_p)\in D_\Lambda(\bA)$  be any adelic point on $D_\Lambda$, then
\[\Ba\Lambda{\Lambda'}=\prod_p \Ba\Lambda{\Lambda'}_p,\quad
\Ba\Lambda{\Lambda'}_p=\prod_{i=1}^3 \brBig{L_i(P_p), d_i'}_p\]
where $p$ runs over all places of $\bQ$ and $\brbig{\cdot,\cdot}_p$ is the Hilbert symbol at $\bQ_p$.

Note that skew-symmetry over $\bF_2$ is also symmetry, so the left kernel and right kernel of the Cassels pairing $\Ba{}{}$ are the same. An important property of Cassels pairing says that the kernel of the Cassels pairing is
\[\Im\Sel_4(\En)/\En(\bQ)[2]\]
where $\Im\Sel_4(\En)$ is associated to the long exact sequence derived from\\
$0\to E[2]\to E[4]\oset{\times 2}\to E[2]\to 0$.

In fact almost all local Cassels pairing are $1$.
\begin{lemma}[Cassels \cite{Cassels1998CasPairing} Lemma 7.2]\label{lem-Cas} The local Cassels pairing $\Ba{ }{ }_p=+1$ if $p$ satisfying
\begin{enumerate}
\item[(1)] $p\not=2,\infty$;
\item[(2)] The coefficients of $H_i$ and $L_i$ are all integral at $p$ for $ i=1,2,3$;
\item[(3)] Modulo $D_\Lambda$ and $L_i$ by $p$, they define a curve of genus $1$ over $\bF_p$ together with tangents to it.
\end{enumerate}
\end{lemma}

\section{Gauss Genus Theory}
In this section, we will introduce Gauss genus theory: classical theory, $4$-rank and Redei matrix, $8$-rank and higher Redei matrix.

Let $K$ be an imaginary quadratic number field with fundamental discriminant $D$ and ideal class group $\cA$( written   additively). The classical Gauss genus theory characterize $\cA[2]$ and $2\cA$, where $\cA[2]$ consists of all those ideal classes killed by $2$ and $2\cA$ denotes all those ideal  classes which are squares.

To state Gauss genus theorem, we give some notations: let $p_1,\cdots, p_t$ be the different prime divisors of $D$, $\cO_K$ the   ring of algebraic integers of $K$. $N=N_{K/\bQ}$ is the norm from $K$ to $\bQ$  and $\alpha_0$ is the integer $\frac{D+\sqrt D}2$, whence we have $\cO_K=\bZ+\bZ\alpha_0$. In this section  we use $(a,b)$ to denote  the greatest common divisor of $a,b$ if $a,b\in \bZ$  else   the ideal generated by $a,b$ if $a,b\in K$. We  also use $(c_1,\cdots,c_k)$ to denote a vector in $\bF_2^k$. We hope this will not cause confusion since we can   tell apart them easily from the context.
\begin{prop}\label{gauss-cla}\quad
\begin{enumerate}
\item[(1)] $\cA[2]$ is an elementary abelian $2$-group generated by the ideal classes $\idealc{p_i}$ and $\rank_{\bF_2} \cA[2]=t-1$;
\item[(2)] For an ideal $\fa$, then $[\fa]\in 2\cA$ if and only if  there is a non-zero integer $z$ and $a\in\fa$ such that: $z^2\cdot N\fa=Na$.
\end{enumerate}
\end{prop}

For the proof of Proposition \ref{gauss-cla} we refer to  Chapter 7 of Hecke \cite{heckeGTM77}, there is also an elementary proof see Nemenzo-Wada \cite{nemenzo1992}.

\subsection{$4$-rank and Redei matrix}

To deal with $4$-rank and $8$-rank, we will prove some well-known results, these can be seen in some papers( for example Jung-Yue \cite{yueJung2011eightrank} ) but with no  proof. As we mentioned in the introduction, we will relate $\cA[2]\cap 2\cA$ with norm elements of $K^\times$,  solving quadratic homogeneous Diophantine equation and also Redei matrix. Note that we have $h_4(\cA)=\rank_{\bF_2} \cA[2]\cap 2\cA$, so the $4$-rank is reduced to study $\cA[2]\cap 2\cA$. Now we first set up the  $2-1$ correspondence on $\cA[2]\cap 2\cA$.
\begin{prop}\label{gauss} The following map is an epimorphism
$$\theta: \cD(K)\cap N_{K/\bQ}(\Kc) \to \cA[2]\cap2\cA,\quad d\mapsto  \Sqlr{\brlr{d,\alpha_0}} $$
with $\Ker(\theta)=\Brbig{ 1, -D' }$, where   $D'$ denotes  the square-free part of $D$ and $\cD(K)$ denotes all those positive square-free divisors of $D$ with group multiplication \[d_1\odot d_2:=\frac{d_1d_2}{(d_1,d_2)^2}\]
\end{prop}

\begin{proof}
First $\theta$ is well-defined: For any $d\in \cD(K)\cap N_{K/\bQ}(\Kc)$, then $\theta(d)=\idealc d\in \cA[2]$ since $d|D $ as $d\in \cD(K)$. From $d\in N_{K/\bQ}(\Kc)$, we know there is some $\alpha\in \Kc$ such that $d=N\alpha$, then by multiply a suitable positive integer $z$, we can write $z\alpha=xd+y\alpha_0\in (d,\alpha_0)$. Hence from Proposition \ref{gauss-cla} we get $\theta(d)=\idealc d\in 2\cA$ since $(d,\alpha_0)$ has norm  $d$.

Second $\theta$ is a homomorphism: For any $d_i\in \cD(K)\cap N_{K/\bQ}(\Kc)$, if $(d_1,d_2)=1$ then we have the following ideal equalities( note multiplication of ideals is written  additively )
\begin{equation}\label{eq:ideal d1 d2 coprime}
\ideal{d_1}+\ideal{d_2}=\brlr{d_1d_2, \alpha_0^2, d_1\cdot\alpha_0, d_2\cdot \alpha_0 } =\brlr{d_1d_2, \alpha_0}
\end{equation}
Then for general $d_1, d_2$ we have $d_i=\frac{d_i}{(d_1,d_2)}\cdot (d_1,d_2)$, whence from equation (\ref{eq:ideal d1 d2 coprime}) we have
\[(d_i,\alpha_0)=\brBig{\frac{d_i}{(d_1,d_2)}, \alpha_0}+\brBig{(d_1,d_2), \alpha_0}\]
Consequently
\[(d_1,\alpha_0)+(d_2,\alpha_0)=2\brBig{(d_1,d_2), \alpha_0}+\sum_{i=1}^2 \brBig{\frac{d_i}{(d_1,d_2)}, \alpha_0}=2\brBig{(d_1,d_2), \alpha_0}+(d_1\odot d_2,\alpha_0) \]
Thus we have
$$\theta(d_1)+\theta(d_2)=\theta(d_1\odot d_2)$$
Hence $\theta$ is a group homomorphism.

Now we show that $\theta$ is surjective: It is clear that from genus theory ( see Proposition \ref{gauss-cla} ),  every element of  $\cA[2]$ is of the form $\idealc d$ with $d$ a positive square-free divisor $D$. So for any   $[\fa]\in \cA[2]\cap 2\cA$, there is a $d\in \cD(K)$, such that $[\fa]=\idealc d$. Since $\idealc d \in 2\cA$, then by Proposition \ref{gauss-cla}  there is a non-zero integer $z$ and an $\alpha\in \ideal d$ such that
$$z^2d=z^2\cdot N_{K/\bQ}\Brbig{  \ideal d}=N\alpha$$
Therefore $d\in N_{K/\bQ}(\Kc)$.

We begin to determine the kernel of $\theta$: For any $d\in \Ker(\theta)$, then $d|D$ is a positive square-free integer such that $d=N(\beta)$ and $\ideal d=(\alpha)$ is a principal ideal.  Thus there are integers $x,y\in \cO_K$ such that
\[\alpha=xd+y\alpha_0\]
Let $x=m+n\alpha_0, y=m'+n'\alpha_0\in \cO_K$ with $m,n,m',n'\in\bZ$, then
$$\alpha=ud+v\sqrt D$$
where
\[u=m+\frac D{2d}\cdot (nd+m')+n' \cdot\frac{D(D+1)}{4d}, \quad v=\frac{nd+n'D+m'}2\]
Hence $N\alpha=d\brBig{du^2+\frac{-D}d v^2}=N\Brbig{\ideal d}=d$, then
\begin{equation}\label{eq:Normalpha}
du^2+\frac{-D}d v^2=1
\end{equation}

(i): If $D\equiv 1 \pmod 4$, then $d$, $\frac Dd$ are odd and $2u,2v\in \bZ$, then from equation (\ref{eq:Normalpha}) we get
\[ d(2u)^2+\frac{-D}d(2v)^2=4 \]
Note that the Diophantine equation $ar^2+bs^2=4$ with $a,b$ positive odd integer and $ab\equiv 3(\mod 4)$ has integer solution only if $a=1$ or $b=1$, hence   $d=1$ or $\frac{-D}d=1$, thus we have two solutions $d=1$ or $-D$.\\

(ii): If $4|D$, we assume that $D\not=-4$, else it is trivial. Now we divided into two cases according to $d$'s parity:

If $d$ is odd, then $u, 2v\in \bZ$ and $4d|D$, thus from equation (\ref{eq:Normalpha}) we get \[du^2+\frac{-D}{4d}(2v)^2=1\]
Then from the solvability of the Diophantine equation $ar^2+bs^2=1$ with $a,b$ positive odd, we get $d=1$ or $\frac{-D}{4d}=1$, the latter solution requires that $4||D$ since $d$ is odd.

If $d$ is even, then $2u,2v\in \bZ$, and $2|d, 4\nmid d, 2\Big|\frac{-D}d$, thus from equation (\ref{eq:Normalpha}) we get:
\[d(2u)^2+\frac{-D}d(2v)^2=4\]

Note the Diophantine equation $at^2+bs^2=4$ with $a,b$ positive even integer and $4 \nmid  a$ can have integral solution only if $a=b=2$ or $b=4$: the first case requires that $D=-4$ which has been excluded; the latter case implies that $d=\frac{-D}4$ and requires $8|D$ since $d$ is even.

In summary,  we always have $\Ker(\theta)=\Brbig{ 1, -D' }$. This completes the proof.
\end{proof}

Now we are ready to introduce the Redei matrix $R$ ( called reduced Redei matrix in Jung-Yue \cite{yueJung2011eightrank} ) of $K$: first we assume that if $2|D$, then $p_t=2$. Then the Redei matrix $R=(r_{ij})$ is a $(t-1)\times t$ matrix over $\bF_2$ defined by:
\[ r_{ii}=\ALeg{D/p_i^*}{p_i}, \quad
r_{ij}=\ALeg{p_j}{p_i} \quad \text{if } i\not=j, \quad {where } \quad p_i^*=(-1)^{\frac{p-1}2}p_i\]

Note $R$ is very similar to the matrix $A$ occurring in the definition of Monsky matrix. Now we can relates $R$ to $h_4(n)$ via the following isomorphism which is parallel to the Monsky matrix.
\begin{prop}\label{gauss-Redei} The following is an isomorphism:
\[\cD(K)\cap N(\Kc)\to \BrBig{ X\in \bF_2^t\Big| RX=0},\quad  d\mapsto X_d:=\brBig{v_{p_1}(d), \cdots, v_{p_{t}}(d)}^T\]
with inverse map given by $(x_1,\cdots,x_t)^T\mapsto \prod_{i=1}^t p_i^{x_i}$.
\end{prop}

This Proposition is a special case of Proposition \ref{gauss-Redei-c-version} with $c=1$. Hence from Proposition   \ref{gauss} and \ref{gauss-Redei}  we know that $h_4(\cA)=t-1-\rank_{\bF_2} R$.

To deal with $4$-rank of $K$, these propositions are sufficient. But for  $8$-rank, we have to generalize Proposition \ref{gauss-Redei} to general $c$, see the following:
\begin{prop}\label{gauss-Redei-c-version} Let $c$ be a positive  odd integer coprime with $D$ such that for any prime divisor $p $ of $c$, we have
$\Leg Dp=1$. Then the following is an isomorphism:
\begin{eqnarray*}
\BrBig{ d\in \cD(K) \Big| dc\in N_{K/\bQ}(\Kc) } &\to & \BrBig{ X\in \bF_2^t\Big| RX=C
 }  \quad \text{ where }\; C=\brlr{ \ALeg c{p_1}, \cdots, \ALeg c{p_{t-1}}}^T  \\
d &\mapsto& X_d:=\brBig{v_{p_1}(d), \cdots, v_{p_{t}}(d)}^T
\end{eqnarray*}
and the inverse map is given by $(x_1,\cdots,x_t)^T\mapsto \prod_{i=1}^t p_i^{x_i}$.
\end{prop}

\begin{proof} For $d\in \cD(K)$, note that $dc\in N_{K/\bQ}(\Kc)$ if and only if the following equation:
\begin{equation}\label{eq:gauss-solvbility-d}
dcz^2=x^2-Dy^2
\end{equation}
is solvable over $\bZ$. Hence   modulo equation (\ref{eq:gauss-solvbility-d}) by odd prime divisors of $D$, we see the solvability of equation (\ref{eq:gauss-solvbility-d}) implies that:
\begin{enumerate}
\item[(1)]: $\Leg{-cD/d}p=1$ for $p\Big|d$ odd prime divisor;
\item[(2)]: $\Leg {cd}p=1$ for $p\Big|\frac{-D}d$ odd prime divisor.
\end{enumerate}
But the inverse direction is also true, since we have
\begin{lemma}
 The local conditions (1),(2) also implies the solvability of the Diophantine equation (\ref{eq:gauss-solvbility-d}) under the hypothesis of Proposition \ref{gauss-Redei-c-version}.
\end{lemma}
\begin{proof} of the Lemma:

By Hasse-Minkovski principle, it suffices to show that (\ref{eq:gauss-solvbility-d}) is locally solvable everywhere. For $p \nmid  2cD$, this is obvious. For $p|c$, this local solvability follows from the hypothesis of Proposition \ref{gauss-Redei-c-version}. For odd $p|D$, this is just the local conditions (1),(2). So we just need  to show that (\ref{eq:gauss-solvbility-d}) is solvable at $\bQ_2$. Whence it suffices to show that the Hilbert symbol $(cd,D)_2=1$.

Let $d_0$ and $\tdd$ be the odd part of $d, \frac{-D}d$ respectively and $2^r=d/d_0$.
If we use $s$ to denotes the $2$-adic valuation of $D$, then \[(c,D)_2=(c,d_0\tdd)_2\cdot\Leg{-2^{s}}c =\Leg{-2^{s}}c\cdot \Leg{d_0\tdd}c\cdot \Leg c{d_0\tdd}=\Leg{D}c\cdot \Leg{c}{d_0\tdd}=\Leg{c}{d_0\tdd}\]
where we have used $\Leg{D}{c}=1$ which follows from  the hypothesis of Proposition \ref{gauss-Redei-c-version}.

Whence we are reduced to show $(d,D)_2=\Leg{c}{d_0\tdd}$. Note the conditions (1) and (2) imply that: $\Leg{2^{s-r} c\tdd}{d_0}=1, \quad \Leg {2^r cd_0}\tdd=1$.
From either $d_0$ or $-d_0$ is congruent to $1 \pmod 4$,   we have $(d_0,-d_0)_2=1$. Therefore
\begin{eqnarray*}
(d,D)_2&=&\brBig{2^rd_0, -2^sd_0\tdd}_2=\brbig{2^r,-2^sd_0\tdd}_2\cdot \brbig{d_0, -2^sd_0\tdd}_2  \\
&=&\brbig{2^r, d_0\tdd}_2\cdot \brbig{d_0, 2^s}_2\cdot \brlr{d_0, -d_0}_2\cdot (d_0, \tdd)_2   \\
&=&\Leg{2^r}{d_0\tdd}\cdot \Leg{2^s}{d_0}\cdot \Leg{d_0}\tdd\cdot \Leg\tdd{d_0}  \\
&=&\Leg{2^{r+s}\tdd}{d_0}\cdot \Leg{2^rd_0}\tdd = \Leg c{d_0\tdd}
\end{eqnarray*}
This completes the proof of the Lemma.
\end{proof}

Thus the expected isomorphism is reduced to show that the local conditions (1) and (2) are equivalent to $RX_d=C$. Let $X=(x_1,\cdots,x_t)^T\in \bF_2^t$ such that $RX=C$, then for any $1\le i\le t-1$ we have
\begin{equation}\label{eq:Redeiequiv}
\sum_{j\not=i} x_j\ALeg{p_j}{p_i}+x_i\ALeg{D/p_i^*}{p_i}=\ALeg c{p_i}
\end{equation}

If $x_i=0$, then $p_i\nmid d:=\prod_{l=1}^t p_l^{x_l}$, and equation (\ref{eq:Redeiequiv}) is equivalent to   $\Leg{cd}{p_i}=1$. If $x_i=1$, then $p_i \mid d$, the equation (\ref{eq:Redeiequiv}) is equivalent to $\Leg{-cD/d}{p_i}=1$.

For $i=t$: we may assume that $p_t\not=2$, whence $D=-p_1\cdots p_t\equiv 1(\mod 4)$. Then add all equations in (\ref{eq:Redeiequiv}) we get
\begin{equation}\label{eq:the case pt}
\sum_{i=1}^{t-1} \sum_{j\not=i} x_j \ALeg{p_j}{p_i}+\sum_{i=1}^{t-1} x_i \ALeg{D/p_i^*}{p_i}=\sum_{i=1}^{t-1} \ALeg c{p_i}
\end{equation}

Then for this equation we have its left hand side is
\begin{eqnarray*}
&&\sum_{i=1}^{t-1} \sum_{j=1,j\not=i}^t x_j \ALeg{p_j}{p_i} +
\sum_{i=1}^{t-1}x_i\sum_{j=1, j\not=i}^{t-1} \ALeg{p_j}{p_i}+\sum_{i=1}^{t-1} x_i \ALeg{p_t}{p_i}  \\
&=& x_t\sum_{i=1}^{t-1}\ALeg{p_t}{p_i}+\sum_{i=1}^{t-1} x_i \ALeg{p_t}{p_i}+
\sum_{i=1}^{t-1}x_i\sum_{j=1, j\not=i}^{t-1} \BrBig{ \ALeg{p_j}{p_i} +\ALeg{p_i}{p_j} }\\
&=&x_t\sum_{i=1}^{t-1}\ALeg{p_t}{p_i}+
\sum_{i=1}^{t-1} x_i \Brlr{ \ALeg{p_i}{p_t}+\sum_{j=1, j\not=i}^{t} \BrBig{ \ALeg{p_j}{p_i} +\ALeg{p_i}{p_j} } }
\end{eqnarray*}
If we use  $p_1\cdots \hp_i\cdots p_t$ to denote the product $p_1\cdots  p_t$ with $p_i$ omitted, then the left hand side of equation (\ref{eq:the case pt}) equals
\[ x_t \ALeg{p_t}{p_1\cdots p_{t-1}}+ \sum_{i=1}^{t-1} x_i \BrBig{\ALeg{p_i}{p_t}+\ALeg{p_i}{p_1\cdots \hp_i\cdots p_t}+\ALeg{p_1\cdots \hp_i\cdots p_t}{p_i}}
=x_t\ALeg{D/p_t^*}{p_t}+\sum_{i=1}^{t-1}x_i\ALeg{p_i}{p_t}\]
where for the last equation we   used the quadratic reciprocity law and $p_1\cdots p_t\equiv 3(\mod 4)$, as then one of $p_i$ and $p_1\cdots \hp_i\cdots p_t$ must be congruent to $1 \pmod 4$. Since $\Leg Dc=1$, then
$$1=\Leg{-p_1\cdots p_t}c=\Leg{-1}c\cdot \Leg{p_1\cdots p_t}c=\Leg{c}{p_1\cdots p_t}$$

Hence the right hand side of the equation (\ref{eq:the case pt}) equals $\ALeg c{p_t}$, thus as what we do for equation (\ref{eq:Redeiequiv}) we get the local solvability is also true for $p_t$.

Hence, the local conditions on $d$ is satisfied if and only if $RX_d=C$. This finishes the proof.
\end{proof}

\subsection{$8$-rank}

In this subsection, we will use classical Gauss genus theory and $4$-rank to derive $8$-rank $h_8(K)=\rank_{\bF_2} \cA[2]\cap 4\cA$ of $K=\bQ(\sqrt{-n})$: this is equivalent to determine those elements in $\cA[2]\cap 2\cA$ also lie  in $4\cA$.

In this subsection, for the quadratic field, we always assume that  $n=p_1\cdots p_k$ is a  positive square-free integer with all prime factors $p_i\equiv 1  \pmod 4$.

For any $2^rd\in \cD(K)\cap N_{K/\bQ}(\Kc)$ with $d|n$  and $r=0$ or $1$ such that $\theta(2^rd)$ is non-trivial, then from Proposition \ref{gauss}, \ref{gauss-Redei} we know that the following Diophantine equation
\begin{equation}\label{eq:z2=dx2+d'y2}
2^rz^2=dx^2+\frac nd y^2
\end{equation}
is solvable over $\bZ$, and let $(a,b,c)=(x,y,z)$ be a positive primitive integer solution,  then $c$ is odd, as else $c$ is even, then $a,b$ must be odd, since $(a,b,c)$ is a primitive solution, then modulo equation (\ref{eq:z2=dx2+d'y2}) by $4$ we get: $0\equiv d+\frac nd \equiv 2 \pmod 4$, since every divisor of $n$ congruent to $1 \pmod 4$, hence we get a contradiction, which implies that $c$ is odd.

Then $[\fd]=[\fc]^2$ where $\fd=\brlr{2^rd,\alpha_0}$ and $\fc$ is a certain integral ideal over $c$:  From $(c,D)=1$ we know that any prime $p|c$ is unramified in $K$, moreover $p$ splits in $K$, as else $(p)$ is   a prime ideal  in $K$, by prime ideal decomposition we know either $(p)\Big|(\alpha)$ or $(p)\Big|(\overline\alpha)$ since we have $2^rdc^2=N\alpha$ with $\alpha=da+b\sqrt{D'}\in \cO_K$, but either $(p)\Big|(\alpha)$ or $(p)\Big|(\overline\alpha)$  implies that $p|(a,b,c)$ which is impossible since $(a,b,c)$ is a primitive solution. Whence we get any prime divisor of $c$ splits in $K$. Then $(p)=\fp\bar\fp \Big| N(\alpha)$, similarly exactly one of $\fp, \bar\fp$ divide $(\alpha)$, then $(\alpha)=\fd\prod_{p|c}\fp^{2 v_p(c)}$. Therefore $[\fd]=[\fc]^2$ with $\fc$ an integral ideal over $c$.

Now we study when $[\fd]\in 4\cA$:  we note that $[\fd]=[\fc]^2\in 4\cA$ if and only if there is a  $[\fm]\in\cA[2]$ with norm $m\in \cD(K) $ such that $[\fc]+[\fm] \in 2\cA$,  which is equivalent to \[mcZ^2=X^2-DY^2\] is solvable over $\bZ$ by Proposition \ref{gauss-cla}.  Modulo the equation (\ref{eq:z2=dx2+d'y2}) by any prime divisor $p$ of $c$, we get $\Leg Dp=1$. And $c$ is odd, then by Proposition \ref{gauss-Redei-c-version}  we know that $[\fd]\in 4\cA$ if and only if $C\in \Im R$ where
$C=\brlr{\ALeg c{p_1},\cdots, \ALeg c{p_k}}^T$.

\begin{rem}\label{rem1}
We remark that for a given $d$ as above and another primitive solution $(a',b',c')$ of the equation (\ref{eq:z2=dx2+d'y2}) there is another $C'$, but we may have $C\not=C'$. For example $K=\bQ(\sqrt{-13*17})$, then $R=\left(
                             \begin{array}{ccc}
                               0 & 0 & 1 \\
                               0 & 0 & 0 \\
                             \end{array}
                           \right)
$, whence $h_4(n)=1$ and $13$ represents the only non-trivial element in $\cA[2]\cap 2\cA$, then for the equation $z^2=13x^2+17y^2$ we have two primitive solution $(x,y,z)=(1,2,9)$ and $(4,1,15)$, and correspondingly  $C=(0,0)^T, C'=(1,0)^T$, they are not equal but both lies in $\Im R$, hence $h_8(n)=1$. In fact, by Sage software we have the ideal class group of $K$ is isomorphic to $\bZ/8\bZ\oplus \bZ/2\bZ$.
\end{rem}

Now if $d_1,d_2\in \cD(K)\cap N_{K/\bQ}(\Kc)$ and $C_1, C_2$ are  given as above, we know $d:=d_1\odot d_2 \in \cD(K)\cap N_{K/\bQ}(\Kc) $. Then we have a natural question: does $[(d,\alpha_0)]\in 4\cA$ have any relation with $C_1, C_2$? This is given by:

\begin{prop}\label{prop-C=C1+C2}
Let $K=\bQ(\sqrt{-n})$ be an imaginary quadratic number field with $n=p_1\cdots p_k$   a square-free positive integer such that all $p_i\equiv 1 \pmod 4$,  and $2^{r_1}d_1, 2^{r_2}d_2\in \cD(K)\cap N_{K/\bQ}(\Kc)$ with $d_i|n$ and $r_i\in \Brbig{0,1}$. Let  $2^rd=2^{r_1}d_1\odot 2^{r_2}d_2$ with $d|n$. Assume that $(a_i, b_i,c_i)=(x,y,z)$  be positive primitive integer solution of
\begin{equation}\label{eq:higherRedei-di-norm}
 2^{r_i}z^2=d_ix^2+\frac n{d_i} y^2
\end{equation}
and  $C_i:=\brlr{\ALeg{c_i}{p_1}, \cdots, \ALeg{c_i}{p_k}}^T, i=1,2$. If we denote $\fd=(2^rd,\alpha_0)$ and $R$ the Redei matrix of $K$, then
$ [\fd]\in 4\cA$ if and only if $ C_1+C_2\in \Im R$.
\end{prop}
\begin{proof} Let $\alpha_i=d_ia_i+b_i\sqrt{-n}\in \cO_K$, then $2^{r_i}d_ic_i^2=\alpha_i\overline\alpha_i$, thus we have
\begin{equation}\label{eq:norm-alpha-i}
2^{r_1+r_2}d_1d_2(c_1c_2)^2=N_{K/\bQ}(\alpha_1\alpha_2)=a^2+nb^2
\end{equation} with $a=d_1d_2a_1a_2-b_1b_2n, \;
b=b_1d_2a_2+b_2d_1a_1$. Then we want to  derive a primitive solution of $2^rz^2=dx^2+\frac ndy^2$ with $z=\frac{c_1c_2}{c_0^2}$ from  equation (\ref{eq:higherRedei-di-norm}), then this will finish our proof  by Proposition \ref{gauss-Redei-c-version}. Before   proceeding further, we introduce a notation: for a prime $p$ and an integer $m$, the notation $p^r||m$  means that $p^r|m$ but $p^{r+1} \nmid  m$. We claim that:

\emph{If a prime $p|\brlr{c_1c_2,a,b}$ then $p^{2s_p}||(c_1c_2,a,b)$, where $s_p:=\min\Brbig{v_p(c_1), v_p(c_2)}$. }\\

Since $p|c_1c_2$ then $p$ is odd and we may assume that $p|c_1$, hence
\begin{equation}\label{eq:p-2-divide dc}
p^2\Big|2^{r_1}d_1c_1^2=\alpha_1\overline\alpha_1
\end{equation}
but $p \nmid   \alpha_1$ by the solution $(a_1,b_1,c_1)$ is primitive and $c_1$ is coprime with $2n$. Similar argument as the paragraph in equation (\ref{eq:z2=dx2+d'y2}) we get $p$ splits in $K$ with $(p)=\fp\bar\fp$ and $\fp\not=\bar\fp$, and exactly one of $\fp, \bar \fp$ divide the principal ideal $(\alpha_1)$, we may assume that $\fp|(\alpha_1)$, so $\bar\fp\nmid (\alpha_1)$. Then from ideal version of equation (\ref{eq:p-2-divide dc}) we know that $\fp^{2 v_p(c_1)}||(\alpha_1)$. But $p$ also divides $a,b$, hence $p|\alpha_1\alpha_2$, whence we have\[\fp\bar\fp | (\alpha_1)(\alpha_2)\]
since $\bar\fp \nmid   (\alpha_1)$, from this we have $\bar\fp|(\alpha_2)$, similarly we have $\bar\fp^{2 v_p(c_2)}||(\alpha_2)$ but $\fp \nmid  (\alpha_2)$. Hence by the prime ideal decomposition of $(\alpha_1\alpha_2)$, we know their $\fp, \bar\fp$ components are $\fp^{2 v_p(c_1)}$ and $\bar\fp^{2v_p(c_2)}$ respectively. From this we get $p^{2s_p}||(c_1c_2,a,b)$:

From this claim we know there is a positive integer $c_0$ such that $(c_1c_2,a,b)=c_0^2$. By the definition of $a,b$ we can easily get $(d_1,d_2)d|a, (d_1,d_2)|b$. Now we deal with $2$-part of $a,b$.

If $r_1=r_2=0$, then from equation (\ref{eq:higherRedei-di-norm}) we know exact one of $a_i, b_i$ is even for $i=1,2$. Since
\begin{equation}\label{eq:parity of ab}
a\equiv a_1a_2-b_1b_2(\mod 2), \quad b\equiv a_1b_2+a_2b_1 (\mod 2)
\end{equation}
then exactly one of $a,b$ is even.

If exact one of $r_1,r_2$ is $0$, then exact one of $a_1, a_2, b_1, b_2$ is even, then $a,b$ are odd by equation (\ref{eq:parity of ab}).

If $r_1=r_2=1$, then $r=0$, and all of $a_1, a_2, b_1, b_2$ are odd, hence from equation (\ref{eq:parity of ab}) we get $a,b$ are even. If $4|(a,b)$, then modulo $a,b$ by $4$ respectively we have
\[a_1a_2\equiv b_1b_2(\mod 4), \quad a_1b_2\equiv -a_2b_1(\mod 4)\]
Then we get
\[a_2b_2\equiv a_1^2a_2b_2\equiv -b_1^2a_2b_2\equiv -a_2b_2(\mod 4)\]
Then $2|a_2b_2$ which is impossible.

Thus in any case  we have $2^{r_1r_2}||(a,b)$. So if we denote\[a'=\frac{a}{c_0^22^{r_1r_2}\cdot d(d_1,d_2)}, \quad b'=\frac b{c_0^22^{r_1r_2}\cdot(d_1,d_2)}, \quad c'=\frac{c_1c_2}{c_0^2}\]
then from equation (\ref{eq:norm-alpha-i}) we get $(x,y,z)=(a',b',c')$ is a primitive positive integer solution of \[2^rz^2=dx^2+\frac{n}dy^2\]
Whence the proof is finished by Proposition  \ref{gauss-Redei-c-version}.
\end{proof}
\subsection{Higher Redei matrix}

In this subsection, we will define a higher Redei matrix $R^*$, which is analogous to Redei matrix, for example we also have a $2$ to $1$ epimorphism from $\ker(R^*)$ to $\cA[2]\cap 4\cA$.

In this subsection ,we let $K=\bQ(\sqrt{-n})$ be an imaginary quadratic number field, where $n=d_1\cdots d_k$ is  a square-free positive integer with  all prime factors congruent to $1 \pmod 8$  such that
\begin{enumerate}
\item[(1)] $h_4(d_i)=1,  1\le i\le k$;
\item[(2)] $\Leg{p_{i_1}}{p_{i_2}}=1$ for any prime divisor $p_{i_j}$ of $d_{i_j}$ with $i_1\not=i_2$.
\end{enumerate}

Since all prime divisors of $n$ are congruent to $1 \pmod 8$, then the Redei matrix $R$ is \begin{eqnarray*}
R=\brBig{A\; 0}, \quad A=\diag(A_{d_1}, \cdots, A_{d_k})
\end{eqnarray*}
as $n$ satisfying (2), where $A_{d_i}$ denotes the corresponding $A$-matrix in defining the corresponding Monsky matrix of $d_i$. From $h_4(d_i)=1$ and all prime divisor of $d_i$ are congruent to $1 \pmod 8$ we get $\rank A_{d_i}=\omega(d_i)-1$ by Proposition \ref{gauss} and \ref{gauss-Redei}, where $\omega(d_i)$ denotes the number of prime divisors of $d_i$.

If we define $d_{k+1}:=2$, then under $\theta$,  the image of $d_1,\cdots,d_{k-1},d_{k+1}$  form a base of $\cA_n[2]\cap 2\cA_n$ by Proposition \ref{gauss} and \ref{gauss-Redei}. Let $\fd_i=(d_i,\alpha_0)$ for $1\le i\le k+1$. Then the problem of determining whether  $[\fd_i]\in 4\cA$ is reduced to determine whether $C_i\in \Im R$ according to the argument before Remark \ref{rem1} of \S3.2. Here $C_i$ is defined to be $\brlr{\ALeg{c_i}{p_1},\cdots, \ALeg{c_i}{p_{\omega(n)}}}^T$, where  $(a_i,b_i,c_i)=(x,y,z)$ is  a fixed primitive positive integer solution of
\[z^2=d_i x^2+\frac{n}{d_i} y^2\]
for $1\le i\le k-1$, and $(a_k,b_k,c_k)$ is the primitive solution derived similarly as in Proposition \ref{prop-C=C1+C2} from those $(a_i,b_i,c_i)$ chosen above since $\theta(d_k)=\prod_{i=1}^{k-1}\theta(d_i)$, while  $(a_{k+1},b_{k+1},c_{k+1})=(x,y,z)$ is a fixed  primitive positive integer solution of\[2z^2=x^2+ny^2\]

Since a vector $y\in \bF_2^{\omega(n)}$ lies in $\Im R$ if and only if
$$(0,\cdots,0,\uset{\omega(d_i)}{\underbrace{1,\cdots,1}},0\cdots,0)y=0$$
Then from this and Proposition \ref{gauss-Redei-c-version} we know $[\fd_i]=[(d_i,\sqrt{-n})]\in 4\cA$ if and only if $\ALeg{c_i}{d_j}=1$ for any $1\le j\le k$. Hence by Proposition \ref{prop-C=C1+C2} we get that any $d=\prod_1^{k+1} d_i^{x_i}$ contributes to $h_8(n)$ if and only if $\sum_{i=1}^{k+1} x_i\ALeg{c_i}{d_j}=0$ for $1\le j\le k$.  This motivates us to define a higher Redei matrix $R^*$: it is a $k\times (k+1)$ matrix over $\bF_2$  given by
\[R^*=\brBig{A^*\Big|\; B^*} \]
with $A^*=(a^*_{ij})_{k\times k},\; a^*_{ij}=\ALeg{c_j}{d_i}$ and $ B^*=\brBig{\ALeg{c_{k+1}}{d_1},\cdots, \ALeg{c_{k+1}}{d_k}}^T$.

Similar to the Redei matrix, the sum of every row elements of $A^*$ is zero  by the choice of $c_k$ and Proposition \ref{prop-C=C1+C2}. And  we also get a $2$ to $1$  epimorphism as in Redei matrix:
\begin{equation}\label{eq:R*-8rank}
\BrBig{ x\in \bF_2^{k+1} \Big| R^*x=0}\to \cA[2]\cap 4\cA,\quad  x\mapsto \theta\brBig{\prod_{i=1}^{k+1} d_i^{x_i}}
\end{equation}
follows from above argument. Thus similar  as $4$-rank and Redei matrix  we get
\[h_8(K)=k-\rank R^*\]
\begin{rem}
It seems as if $A^*$ depends on the choice of $(a_i,b_i,c_i)$, but from Cassels pairing we will see that $A^*$ doesn't depend on the choice, so $A^*$ is intrinsic.
\end{rem}

\section{Proof of Main Theorems}
In this section, we will prove our main theorems according to the strategy explained in the introduction.

\subsection{ First Main Theorem }

For a positive square-free integer $n=p_1\cdots p_k\equiv 1 \pmod 8$ with all prime factors congruent to $1 \pmod 4$, the corresponding Monsky matrix of $\Sel_2(\En)$ and Redei matrix of $\bQ(\sqrt{-n})$ are of the form
\[M=\left(
      \begin{array}{cc}
        A+D_2 & D_2 \\
        D_2 & A+D_2 \\
      \end{array}
    \right), \quad R=\brBig{ A \Big| B }
\] where $B=(b_1,\cdots,b_k)^T$ with $b_i= \ALeg2{p_i}$ and $A$ is a $k\times k$ symmetric matrix over $\bF_2$ with all row sum $0$, hence all $A$'s column sum are also $0$. And we have $h_4(n)=k-\rank R$.

According to the strategy in the introduction, we begin with the first step: relate $s_2(n)$ with $h_4(n)$ via Monsky matrix and Redei matrix.

\begin{lemma}
For a positive square-free integer $n=p_1\cdots p_k\equiv 1\pmod 8 $ with all $p_i\equiv1 \pmod 4$,
then $s_2(n)=2$ if and only if $  h_4(n)=1$. And if this is satisfied,  then the corresponding pure-$2$ Selmer group $\Sel_2(\En/\bQ)/\En(\bQ)[2]$ is generated by
\[ {(n,n,1)}, \quad { (d,d^*,dd^*)} \]
where $ d^*=d $ if $\rank A=k-2$ and else $d^*=n/d$ with  $d=\prod p_i^{x_i}$, here $x=(x_1,\cdots,x_k)^T\in \bF_2^k$ is a non-trivial solution of $Ax=0$ with $x\not=x_0$ if $\rank A=k-2$ and else $x$ is a solution of $Ax=B$, here $x_0=(1,\cdots,1)^T$.
\end{lemma}
\begin{proof} By properties of Monsky matrix and Redei matrix, this is equivalent to show that the rank of $M$ is $2k-2$ if and only if that of $R$ is $k-1$.

If we denote $R_i$ to be the $(k-1)\times (k+1)$ matrix obtained from $R$ by deleting the $i$-th row vector, then by elementary linear transforms, the following matrices have same rank over $\bF_2$:
\[M, \quad  M':=\Ma A{D_2}{}A, \quad M'':= \Ma{R_k}*{}{R_1^T}\]
where $M''$ derives from $M'$ by: adding the first $k-1$ row to the $k$-th row, then adding the last $k-1$ column to the $(k+1)$-th column, and then moving the $k$-th  row as the last row.

If $\rank R=k-1$:  from $A$ is symmetric and $\sum_{i=1}^k b_i=\ALeg2n$ with $n\equiv 1(\mod 8)$  we get the sum of all row vectors of $A$ and $B$ are $0$, hence the only non-trivial linear dependence of $R$ is the sum of all row vectors is $0$, whence we get rank of $R_i$ is also $k-1$. From this we know that $2k-2\le \rank M''\le 2k-1$, but the rank of $M$ and $M''$ are equal and rank $M$ must be even since $n\equiv 1 \pmod 8$ by the last line of \S2.1, hence $\rank M=2k-2$.

If $\rank M=2k-2$, then so is $\rank M''$, since $R_k$ has only $k-1$ rows, hence from the rank of $M''$ we infers that $\rank R_1^T\ge k-1$, but we know that $\rank R_1\le \rank R\le k-1$ since the sum of all row vectors of $R$ are $0$. This force $R$ has rank $k-1$. So $s_2(n)=2$ if and only if $h_4(n)=1$.\\

Now assume that $h_4(n)=1$:   we suffice to find the corresponding $4$ representatives of $\Sel_2(\En)/\En(\bQ)[2]$. By Monsky matrix, we reduce to find $y,z \in \bF_2^k$  such that
$$M\left(
  \begin{array}{c}
    y\\z
  \end{array}
\right)=
\left(
           \begin{array}{cc}
             A+D2 & D_2 \\
             D_2 & A+D_2 \\
           \end{array}
         \right)\left(
  \begin{array}{c}
    y\\z
  \end{array}
\right)=
0$$
From this we have: $Ay+D_2(y+z)=0, A(y+z)=0$, thus $y+z\in\ker A$. Now we divide into two cases according to $A$'s rank:

If $\rank A=k-2$, and if $y+z=0$ then we have $Ay=0$, thus in this case $y=z$ with $Ay=0$ are the $4$ solutions. Hence the following $2$ elements generates the pure $2$-Selmer group $\Sel_2(\En)/\En(\bQ)[2]$:
\[(d,d,1), \quad (n,n,1)\]

If $\rank A=k-1$, thus $y+z=0$ or $x_0:=(1,\cdots,1)^T$. If $y+z=0$, then we have two solutions $y=z=0$ or $ y=z=x_0$; if $y+z=x_0$, then $Ay=D_2\cdot x_0=B$, and $B$ is indeed in $\Im A$, as $\rank(A|B)=\rank A=k-1$ thus $\Brlr{y,z}=\Brlr{x,x_0-x}$ with $Ax=B$. Hence the following $2$ elements generates $\Sel_2(\En)/\En(\bQ)[2]$:
\[(d,d^*,dd^*), \quad (n,n,1)\]

Hence the Lemma is proved.
\end{proof}

Now we begin to prove Theorem \ref{mainthm1}:
\begin{proof} From Lemma above, we get $ s_2(n)=2$ if and only if $h_4(n)$=1. Now be begin the second step. From the following exact sequence:
$$0\to \En[2]\to \En[4]\oset 2\to \En[2]\to0$$ we have the corresponding long exact sequence:
\[0\to \En[2](\bQ)/2\En[4](\bQ)\to \Sel_2(\En) \to \Sel_4(\En)\to \Sel_2(\En)\]
And if we denote $\Im \Sel_4(\En)$ to be the image of $\Sel_4(\En)$ in the last map of above long exact sequence, then we have $\rank_\bZ \En(\bQ)=0, \; \Sha(\En/\bQ)[2^\infty]\simeq \brbig{\bZ/2\bZ}^2$ if and only if $\#\Sel_2(\En)=\#\Sel_4(\En)$, which is equivalent to $\#\Im\Sel_4(\En)=\#\En(\bQ)[2]$ by above long exact sequence. Then  by Cassels pairing, this holds   if and only if the Cassels pairing on the pure $2$-Selmer group $\Sel_2(\En)/\En(\bQ)[2]$  is non-degenerate. Hence if we denote $\Lambda=(d,d^*,dd^*), \Lambda'=(-1,1,-1)$, then we reduce to show when
$$\Ba\Lambda{\Lambda'}=-1$$ note here $\Lambda'\in (n,n,1)+\En(\bQ)[2]$. \\

Now we are in the third step to  compute $\Ba\Lambda{\Lambda'}$ according to $\rank A$:

(I): The rank of $A$ is $k-2$: we claim that in this case
$$d\equiv 5 \pmod 8$$

From  $h_4(n)=1 $ we get $\rank R=\rank(A|B)=k-1$, but $\rank A=k-2$, this shows that  $B\not\in\Im A$. Since $A$ is symmetric and $\ker A$ is generated by $x_0, x$ which are defined in above Lemma, a vector $y\in\bF_2^k$ lies in $\Im A$ if and only if $x_0^T y=0, x^T y=0$. But $x_0^T \cdot B=\ALeg2n=0$, so we must have $x^T\cdot B=\ALeg2d\not=0$,  thus $d\equiv5(\mod 8)$.

For $\Lambda=(d,d,1)$ the corresponding genus one curve $D_\Lambda$ is:
\[D_\Lambda=
\begin{cases}
H_1:& -nt^2+du_2^2-u_3^2=0\\
H_2:& -nt^2+u_3^2-du_1^2=0\\
H_3:& 2d't^2+u_1^2-u_2^2=0
\end{cases}\] Here  $d':=n/d$.
According to Cassels pairing, we have to chose global points $Q_i$  on $H_i$: for $H_3$ we chose the global point $Q_3=(0,1,1)$, then the tangent plane at $Q_3$ is
$$L_3: u_1-u_2$$
For $H_1$ we chose the global point $Q_1=(b,c,da)$ on $H_1$, where $(a,b,c)$ is a primitive positive integer solution of
\begin{equation}\label{eq:cab-rankk-2}
c^2=da^2+ d'b^2
\end{equation}
The existence of this solution origins from $h_4(n)=1$ and $d$ corresponds to
$R\left(
                    \begin{array}{c}
                      x \\
                      0 \\
                    \end{array}
                  \right)=0
$ by Proposition \ref{gauss}, \ref{gauss-Redei}. We may assume that $2|a$, since else $a$ is odd, $b$ is even and we have a new solution
\[(\tda,\tdb,\tdc):=\brBig{ d'a-2d'b-da, db-2da-d'b, (d+d')c}\]
with $v_2(\tda)\ge2$  and $v_2(\tdb)=v_2(\tdc)=1$, hence by dividing the greatest common divisor, we get a new primitive solution with corresponding $a$ even.  Then the corresponding tangent plane $L_1$ at this point is
\[L_1:  d'bt-cu_2+au_3\]

Hence by Cassels pairing we have to compute
\[\Ba\Lambda{\Lambda'}=\prod_p\brBig{L_1L_3(P_p),-1}_p\]
for any local point $P_v$ on $D_\Lambda(\bQ_p)$, here $p$ includes the real place.

By Lemma \ref{lem-Cas} we only have to compute those $p|2n\infty$. Note all $p|n$ are congruent to $1 \pmod 4$, so $-1$ is a square in $\bQ_p$, thus by Hilbert symbol we have these $\brBig{L_i(P_p),-1}_p=1$. Whence we only have to consider local solutions at $2\infty$, we can chose as follows:

For $p=\infty$: we chose $t=0,   u_1=-u_2=1,  u_3=\sqrt d$, then
$$\brBig{L_1L_3(P_\infty),-1}_\infty=\brBig{2(c+a\sqrt{d}),-1}_\infty=1$$

For $p=2$, we chose $t=2,  u_1=1, u_2^2=1+8d',  u_3^2=d+4n$ and we may assume $u_2\equiv 3(\mod 8)$ as $u_2^2=1+8d'\equiv9(\mod 16)$, so $$\brBig{L_1L_3(P_2),-1}_2=\brBig{(1-u_2)(2bd'-cu_2+au_3),-1}_2=(-2,-1)_2(c,-1)_2=-(c,-1)_2$$
where we  used the fact that $\frac{a u_3}2\equiv bd'(\mod 2)$. Thus  in this case we have
\[\Ba\Lambda{\Lambda'}=-\Leg{-1}c\]
By modulo $c$ of the equation (\ref{eq:cab-rankk-2}) we get $\Leg{-n}c=1$: whence $\Ba\Lambda{\Lambda'}=-1$ if and only if $\Leg cn=1$.

Now we begin our last step is this case: use Gauss genus theory to reduce to $8$-rank of the ideal class group of $\bQ(\sqrt{-n})$. Since $R$ has rank $k-1$ and $x_0^TR=0$, we know a vector $y\in\bF_2^k$ lies in $\Im R$ if and only if $x_0^Ty=0$. From this we get $\Leg cn=1$ is equivalent to $ Rz=C$ has a solution $z\in\bF_2^k$,  this is equivalent to $h_8(n)=1$  by the argument before Remark \ref{rem1} of \S3.2. Hence under the condition $\rank A=k-2$ we know that the Cassels pairing is non-degenerate if and only if $h_8(n)=1\equiv \frac{d-1}4(\mod 2)$ since $d\equiv 5(\mod 8)$. \\

(II): if $\rank A=k-1$, then $d$ corresponds to the solution of $Ax=B$. Since we have
\[R\left(
     \begin{array}{c}
       x \\
       1 \\
     \end{array}
   \right)=Ax+B=0
\] we know that $2d$ is a norm. Then  $(d,d',n)\in(2d,2d,1)+\En(\bQ)[2]$, so we still use $\Lambda$ to denote $(2d,2d,1)$ since this won't effect $\Ba \Lambda{\Lambda'}$, then the corresponding genus $1$ curve $D_\Lambda$ is given by:
\[
\begin{cases}
H_1:& -nt^2+2du_2^2-u_3^2=0\\
H_2:& -nt^2+u_3^2-2du_1^2=0\\
H_3:& d't^2+u_1^2-u_2^2=0
\end{cases}
\]

For $H_3$ we chose the global point $Q_3=(0,1,1)$, then the tangent plane $L_3$ of $H_3$ at $Q_3$ is $$L_3: u_1-u_2$$
Since $2d$ is a norm,  there is a positive primitive solution of the Diophantine equation $$2c^2=da^2+d'b^2$$
So the global point $Q_1: (b,c,ad)$ is on $H_1$, and the tangent plane $L_1$ of $H_1$ at $Q_1$ is \[L_1: d'bt-2cu_2+au_3\]

With the same reason as $\rank A=k-2$,  we only need to consider at $2\infty$. For $p=\infty$ we chose local solution:
\[P_\infty:  t=0,   u_1=-u_2=1, u_3=\sqrt{2d} \]
Then
$$(L_1L_3(P_\infty),-1)_\infty=\brBig{2(2c+a\sqrt{2d}),-1}_\infty=1$$

For $p=2$, we chose local solution
\[P_2: t=1,   u_1=2\ALeg2d,  u_2^2=d'+u_1^2,  u_3^2=n+2du_1^2 \]
such that $cu_2\equiv 1 \pmod 4, 8|(bd'+au_3)$, since we have
\begin{eqnarray*}
(bd')^2-(au_3)^2&=&b^2{d'}^2-a^2n-2a^2du_1^2=(2c^2-da^2)d'-a^2n-2a^2du_1^2  \\
&=& 2(c^2d'-a^2n-a^2du_1^2)\equiv 2(d'-n-du_1^2)\equiv 0(\mod 16)
\end{eqnarray*}
so we may chose $u_3$ such that: $8|(bd'+au_3)$.

\begin{eqnarray*}
\brBig{L_1L_3(P),-1}_2&=&\brBig{(u_1-u_2)(d'b-2cu_2+au_3),-1}_2=(-2,-1)_2(u_1-c,-1)_2\\
&=&(c-u_1,-1)_2=(-1)^{\frac{d-1}4}(c,-1)_2
\end{eqnarray*}

Similar argument as above, we know that $(c,-1)_2=1$ if and only if $\Leg cn=1$, this is equivalent to  $h_8(n)=1$, so Cassels pairing is non-degenerate in this case if and only if $h_8(n)\equiv \frac{d-1}4  \pmod 2$. \\

Summarize these cases together we get: $\rank_\bZ \En(\bQ)=0, \Sha(\En/\bQ)[2^\infty]\simeq \brbig{\bZ/2\bZ}^2$ is equivalent to  $h_8(n)\equiv \frac{d-1}4 \pmod 2$. This completes the proof of  Theorem
\ref{mainthm1}.
\end{proof}

To show  the condition of Li-Tian is also sufficient as claimed in the introduction, we first introduce the quartic residue symbol used in Jung-Yue \cite{yueJung2011eightrank}: for $p$ a prime congruent to $1 \pmod 4$, then there are two primitive Gaussian primes $\pi, \pi'$ over $p$ with $p=\pi\pi'$, hence we have two quartic residue symbol $\Leg{}\pi, \Leg{}{\pi'}$ over $\bZ[i]$ such that their squares are the Legendre residue symbol $\Leg{}p$ over $\bZ$. If $q\in \bZ$ such that $\Leg qp=1$, then $\Leg q\pi=\Leg q{\pi'}=\pm1$, and we denote by $\Leg{q}p_4$ be either   $\Leg q\pi$ or $\Leg q{\pi'}$. Similarly if $d$ is a  positive integer with all prime divisor $p\equiv 1 \pmod 4$ such that  $\Leg qp=1$, then
\[\Leg qp_4:=\prod_{p \mid  d} \Leg qp_4^{v_p(d)}\]

\begin{thm}\label{lem-LiTian}
For $p\equiv 1 \pmod 8$ a prime, then there exists positive integers such that: $p=u^2+8v^2=a^2+16b^2=x^2-32y^2$. Then the following are equivalent:
\begin{enumerate}
\item[(1)] $2 \nmid  v$;
\item[(2)] $2 \nmid  \frac{p-1}8+b$;
\item[(3)] $1+i_p$ is not a $p$-adic square with $i_p\in \bQ_p$ such that $i_p^2=-1$;
\item[(4)] $1+j_p$ is not a $p$-adic square with $j_p\in \bQ_p$ such that $j_p^2=2$;
\item[(5)] $\Leg2p_4=(-1)^{\frac{p-9}8}$;
\item[(6)] $x\equiv 3(\mod 4)$;
\item[(7)] $h_8(p)=0$.
\end{enumerate}
And we define $\delta_p=1$ if $p$ satisfies the above equivalent conditions, else $\delta_p=0$.
For $n$ a square-free positive integer with all prime factors congruent to $1 \pmod 8$ and $h_4(n)=1$, the following are equivalent:
\begin{enumerate}
\item[(i)] $2 \nmid   \delta_n:=\sum_{p|n} \delta_p $
\item[(ii)] $h_8(n)=0$
\item[(iii)] $\rank_\bZ \En(\bQ)=0, \;\Sha(\En/\bQ)[2^\infty]\simeq \brbig{\bZ/2\bZ}^2$.
\end{enumerate}
\end{thm}
\begin{proof}
The existence of $u,v,a,b$ follows from $p$ splits in $\bQ(\sqrt{-2})$ and $\bQ(i)$.  Now we show the existence of $x,y$: since $p$  splits in $\bQ(\sqrt 2)$,   there are $x_0,y_0\in \bZ$ such that
\[p=x_0^2-2y_0^2\]
moreover $x_0$ is odd and $y_0$ is even as $p\equiv 1\pmod 8$. If $4|y_0$ then we define $x=x_0, y=\frac{y_0}4$. Else $4\nmid y_0$,  we can use the fundamental units $1+\sqrt 2$  to lift  $2$-divisibility. Concretely since $(1+\sqrt 2)^{2}=3+2\sqrt2$ has norm $1$, then $$p=(3x_0+4y_0)^2-2\brBig{{2x_0+3y_0}}^2$$
and $4|(2x_0+3y_0)$, then we define $x=3x_0+4y_0, y=\frac{{2x_0+3y_0}}4$. Then we get
$p=x^2-32y^2$.

The equivalence from  (1) to (5) are proved in Li-Tian \cite{LiDeLangTian}.

(5) is equivalent to (6): From
\begin{equation}\label{eq: p=x2-32y2}
p=x^2-32y^2
\end{equation}
we get
$$\Leg xp=\Leg {x^2}p_4=\Leg{-2y^2}p_4=\Leg2p_4\Leg yp=\Leg2p_4$$
where we used $\Leg yp=1$ by the equation (\ref{eq: p=x2-32y2}) modulo $y$.  But if we modulo the equation (\ref{eq: p=x2-32y2}) by $x$ we get
$$\Leg px=\Leg{-2}x=\Leg{-1}x (-1)^{\frac{x^2-1}8}=\Leg{-1}x (-1)^{\frac{p-1}8}$$
since $p\equiv x^2(\mod 32) $ by equation (\ref{eq: p=x2-32y2}) modulo $32$. Then by quadratic reciprocity law we know $\Leg xp=\Leg px$, from this we know (5) is equivalent to (6).\\

(6) and (7) are equivalent because by considering   $1+\sqrt 2$  similar as above
\begin{equation}\label{eq:2as norm element Q(-p)}
-p=N_{\bQ(\sqrt 2)/\bQ}\brBig{(x+8y)+(x+4y)\sqrt 2}=(x+8y)^2-2(x+4y)^2
\end{equation}
then  for the field $\bQ(\sqrt{-p})$, its Redei matrix is the  matrix $0_{1\times 2}$, so we get $2$ is a norm element, and in fact the equation (\ref{eq:2as norm element Q(-p)}) is the corresponding Diophantine equation, see Proposition \ref{gauss}. Hence by  Proposition \ref{gauss-Redei-c-version} with $c=x+4y$ there we get: $h_8(p)=0$ if and only if $\Leg{x+4y}p=-1$. But modulo the equation (\ref{eq:2as norm element Q(-p)}) by $x+4y$ we get $\Leg{-p}{x+4y}=1$, whence $$\Leg{x+4y}p=\Leg p{x+4y}=\Leg{-1}{x+4y}=\Leg{-1}x$$
From these we get  (6) is equivalent to (7).\\

The equivalence of (ii) and (iii) follows from Theorem \ref{mainthm1}. For the equivalence of (i) and (ii): let $x_i, y_i$ be positive integers such that
$p_i=x_i^2-32y_i^2$ as above, then
$$n=p_1\cdots p_k=N_{\bQ(\sqrt2)/\bQ}(x+4y\sqrt 2), \quad x+4y\sqrt2=\prod_1^k (x_i+4y_i\sqrt2) $$ and $x\equiv x_1\cdots x_k (\mod 4)$. Similarly by considering the fundamental unit $1+\sqrt2$, we have:
\[2(x+4y)^2=(x+8y)^2+n\]
Then using Gauss genus theory as above we get: $h_8(n)=0$ if and only if $\Leg{x+4y}n=-1$, we also have  $\Leg {-n}{x+4y}=1$, so
$$\Leg{x+4y}n=\Leg n{x+4y}=\Leg{-1}{x+4y}=\Leg{-1}{x_1\cdots x_k}=(-1)^{\delta_n}$$ From this we get the equivalence of (i) and (ii).

Hence we complete the proof.
\end{proof}

\subsection{ Second Main Theorem}

In this subsection ,we let $n=d_1\cdots d_k$ be a square-free positive integer with all prime factors congruent to $1  \pmod 8$, such that
\begin{enumerate}
\item[(1)] $h_4(d_i)=1,  1\le i\le k$;
\item[(2)] $\Leg{p_{i_1}}{p_{i_2}}=1$ for any prime divisor $p_{i_j}$ of $d_{i_j}$ with $i_1\not=i_2$.
\end{enumerate}

Since all prime divisors of $n$ are congruent to $1(\mod 8)$, hence the corresponding $D_2, D_{-2}$ matrix of $n$ are $0$, thus
\begin{eqnarray*}
M_n&=&\diag\brBig{A, A},\quad R_n=\brBig{A\; 0}
\end{eqnarray*}
with $A=\diag\brBig{A_{d_1}, \cdots, A_{d_k}}$
as $n$ satisfying (2), where $A_{d_i}$ denotes the corresponding $A$-matrix in defining the corresponding Monsky matrix of $d_i$. From $h_4(d_i)=1$ and all prime divisor of $d_i$ are congruent to $1 \pmod 8$ we get $\rank A_{d_i}=\omega(d_i)-1$ by Proposition \ref{gauss} and \ref{gauss-Redei}.\\

Now we get $s_2(n)=2k$: since
$$s_2(n)=2\sum_{i=1}^k \omega(d_i)-\rank M=
2\sum_{i=1}^k\brBig{\omega(d_i)-\rank A_{d_i}}=2k$$
And the following is a base of $\ker M_n$:
\begin{eqnarray*}
\brBig{\uset{\omega(d_1\cdots d_{i-1})}{\underbrace{0,\cdots,0}},
\uset{\omega(d_i)}{\underbrace{1,\cdots,1}},
\uset{\omega(d_{i+1}\cdots d_k)}{\underbrace{0,\cdots,0}},
\uset{\omega(n)}{\underbrace{0,\cdots,0}}}^T,
\brBig{
\uset{\omega(n)}{\underbrace{0,\cdots,0}},
\uset{\omega(d_1\cdots d_{i-1})}{\underbrace{0,\cdots,0}},
\uset{\omega(d_i)}{\underbrace{1,\cdots,1}}, \uset{\omega(d_{i+1}\cdots d_k)}{\underbrace{0,\cdots,0}}
}^T
\end{eqnarray*}
with $1\le i\le k$,  consequently  $\Sel_2(\En)/\En(\bQ)[2]$ has a base consists of
\[\Lambda_i=(1, d_i, d_i),\;\; \Lambda_i'=(d_i, d_i, 1), 1\le i \le k\]

Now we can compute the Cassels pairings of this base, see:
\begin{prop}\label{prop-anySha}
Let $n, \Lambda_i, \Lambda_i'$  as above,  we have the following Cassels pairings:
\begin{eqnarray*}
\Babig{\Lambda_i}{\Lambda_j}&=&\Leg{c_i\gamma_i}{d_j}, \;i\not=j  \\
\Babig{\Lambda_i}{\Lambda_j'}&=&\Leg{c_i}{d_j}=\Leg{\barc_j}{d_i}, \;i\not=j   \\
\Babig{\Lambda_i}{\Lambda_i'}&=&(-1)^{\delta_{d_i}}\Leg{c_i}{d_i}=
(-1)^{\delta_{d_i}}\Leg{\barc_i}{d_i}  \\
\Babig{\Lambda_i'}{\Lambda_j'}&=&\Leg{c_i\barc_i}{d_j}, \;i\not=j
\end{eqnarray*}
Where $c_i $ consists a fixed primitive positive integer solution of
\begin{equation}\label{eq:cCgamma-ci}
c_i^2=d_ia_i^2+d_i' b_i^2
\end{equation}
with  $ d_i'=n/d_i$, and $\gamma_i, \barc_i$ consist   primitive positive integer solutions of
\begin{equation}\label{eq:cCgamma}
\gamma_i^2=d_i\alpha_i^2+2d_i' \beta_i^2, \quad \barc_i^2=d_i\bara_i^2-d_i'\barb_i^2
\end{equation}
respectively such that:
\begin{equation}\label{eq-cond:cCgamma}
d_i'\Big| \brBig{a_i\gamma_i+ c_i \alpha_i , \;\; a_i\barc_i+c_i\bara_i}
\end{equation}
\end{prop}
\begin{rem}
We remark that we can always find solutions of the equations (\ref{eq:cCgamma}) satisfying the conditions (\ref{eq-cond:cCgamma}), we will omit all the subscripts in this remark. If we define $t= (d',  a\gamma  +c \alpha )$, then we have a new solution
\begin{equation}\label{eq:newsolution}
( \tilde{\alpha}, \tilde{\beta}, \tilde\gamma )=
\brBig{-d \alpha  t^2+4d '\beta  t+2 d '\alpha ,\; d \beta  t^2+2d  \alpha  t-2d '\beta ,\; \gamma  (d t^2+2d ') }
\end{equation}
of $z^2=d x^2+2d 'y^2$. Then
\[a\tilde{\gamma}+c\tilde{\alpha}=d t^2(a\gamma-c \alpha)+2d'\brBig{2c\beta t+(a \gamma +c \alpha )}\]By modulo equations (\ref{eq:cCgamma-ci}) and (\ref{eq:cCgamma}) by $d'$ we get $d'\big|(a\gamma+c\alpha)(a\gamma-c\alpha)$.

Then for any $p|t=(d' , a \gamma    +c \alpha )$, we have $v_p(a\tilde{\gamma}+c\tilde{\alpha})\ge 2$ and $v_p(\tilde{\alpha})=1$ as $p \nmid  \alpha$.
For $p\Big|\frac{d '}t$, then $p|(a\gamma-c\alpha)$, therefore $v_p(a\tilde{\gamma}+c\tilde{\alpha})\ge 1$ and $v_p(\tilde{\alpha})=0$. Thus we get a solution of $z^2=d x^2+2d 'y^2$ satisfying the condition (\ref{eq-cond:cCgamma}) by dividing the greatest common divisors  of $\tilde{\alpha}$ and $\tilde{\gamma}$ in (\ref{eq:newsolution}).

Similarly for $(\bara, \barb, \barc)$ we can get a solution satisfying (\ref{eq-cond:cCgamma}). In fact we choose $s=(d ',a \barc +c \bara )$, and we get a solution
\[\brBig{-d\bara s^2+2d'\barb s-d'\bara,\; d\barb s^2-2d\bara s+d'\barb, \; \barc(ds^2-d')}\]
Then do similarly as above.
\end{rem}

Now we are ready to prove Proposition \ref{prop-anySha}:
\begin{proof} For $\Lambda_i=(1,d_i,d_i)$ the corresponding genus one curve $D_{\Lambda_i}$ is:
\[
\begin{cases}
H_1:& -nt^2+d_iu_2^2-d_iu_3^2=0\\
H_2:& -nt^2+d_iu_3^2-u_1^2=0\\
H_3:& 2nt^2+u_1^2-d_iu_2^2=0
\end{cases}\]

Then we chose
$$Q_1=(0,1,1),\quad Q_2=(b_i,c_i, d_ia_i),\quad Q_3=(\beta_i, d_i\alpha_i, \gamma_i )$$
these global points are on $H_i$ follows from equations (\ref{eq:cCgamma-ci}) and (\ref{eq:cCgamma}). Then the corresponding tangent planes are
\begin{eqnarray*}
L_1:&&  u_2-u_3  \\
L_2:&& d_i'b_it-c_iu_3+a_iu_1  \\
L_3:&&  2d_i'\beta_i t+\alpha_i u_1-\gamma_i u_2
\end{eqnarray*}

By Lemma \ref{lem-Cas}, we only need to compute $p|2n\infty$. Since all $d_i$ are bigger than $0$ and congruent to $1 \pmod 8$, hence by properties of Hilbert symbol, they are trivial at $p=2, \infty$. So we only need to consider those $p|n$:

We can chose local solutions $P_p\in D_{\Lambda_i}(\bQ_p)$ as following: For $p\big| d_i'$:
\[\begin{array}{ccc}
t=0,  &u_2=-u_3=1,& u_1^2=d_i
\end{array}\]
such that $p\big| (c_i-a_iu_1),\;p\big| (\alpha_i u_1+\gamma_i) $, as $(c_i-a_iu_1)(c_i+a_iu_1)=d_i'b_i^2$, so we can chose $u_1$ such that $p\big| (c_i-a_iu_1)$, thus we also have $p\big| (\alpha_i u_1+\gamma_i)$ follows from condition (\ref{eq-cond:cCgamma}). For $p\big| d_i$:
\[\begin{array}{ccc}
t=1, & u_1=0, u_3^2=d_i', &u_2=-j_p u_3
\end{array}\] such that $p\big|(d_i'b_i+c_iu_3)$, since we have $(d_i'b_i+c_iu_3)(d_i'b_i-c_iu_3)=-na_i^2$, and $j_p\in \bQ_p$ satisfying $j_p^2=2$ as in Theorem \ref{lem-LiTian}. \\

Now we begin to compute Cassels pairings:

First we compute $\Babig{\Lambda_i}{\Lambda_j}$ with $i\not=j$, then by Cassels pairing we have
$$\Babig{\Lambda_i}{\Lambda_j}=\Babig{(1,d_i,d_i)}{(1,d_j,d_j)}=\prod_p \brBig{L_2L_3(P_p), d_j}_p$$
Note that $\Leg{d_j}p=1$ if $p \nmid   d_j$, so this pairing is trivial at $p| d_j'$, thus we only need consider at those $p|d_j$, and for $p|d_j$:
$$\brBig{L_2L_3(P_p),d_j}_p=\brBig{(c_i+a_iu_1)(\alpha_i u_1-\gamma_i), d_j}_p=\brBig{-c_i\gamma_i, d_j}_p=\Leg{c_i\gamma_i}p$$
 Thus we have:
\[\Babig{(1,d_i,d_i)}{(1,d_j,d_j)}=\Leg{c_i\gamma_i}{d_j} \]\quad\\

Second we compute $\Babig{\Lambda_i}{\Lambda_j'}$ with $i\not=j$, then Cassels pairing implies that
$$\Babig{\Lambda_i}{\Lambda_j'}=\Babig{(1,d_i,d_i)}{(d_j,d_j,1)}=\prod_p\brBig{L_1 L_2(P_p), d_j}_p$$
Since $\Leg{d_j}p=1$ if $p \nmid   d_j$,   this pairing is trivial at $p| d_j'$, therefore we suffice to consider at those $p|d_j$, and for $p|d_j$:
$$\brBig{L_1L_2(P_p),d_j}_p=\brBig{2(c_i+a_iu_1), d_j}_p=\brBig{c_i, d_j}_p=\Leg{c_i}p$$
Whence
\[\Babig{(1,d_i,d_i)}{(d_j,d_j,1)}=\Leg{c_i}{d_j} \]\quad\\

Third we compute $\Babig{\Lambda_i}{\Lambda_i'}$, from Cassels pairing we have $$\Babig{\Lambda_i}{\Lambda_i'}=\Babig{(1,d_i,d_i)}{(d_i,d_i,1)}=\prod_p \brBig{L_1L_2(P_p), d_i}_p$$
From $\Leg{d_i}p=1$ if $p \nmid   d_i$ we know this pairing is trivial at $p| d_i'$, thus we reduce to compute at those $p|d_i$, and for $p|d_i$:
$$\brBig{L_1L_2(P_p),d_i}_p=\brBig{(-j_p-1)u_3(d_i'b_i-c_iu_3), d_i}_p=\brBig{c_i(1+j_p), d_i}_p=(-1)^{\delta_p}\Leg{c_i}p$$
Consequently
\[\Babig{(1,d_i,d_i)}{(d_i,d_i,1)}=(-1)^{\delta_{d_i}}\Leg{c_i}{d_i} \]\quad \\

For $\Lambda_i'=(d_i,d_i,1)$, the corresponding $D_{\Lambda_i'}$ is given by:
\[
\begin{cases}
H_1:& -nt^2+d_iu_2^2-u_3^2=0\\
H_2:& -nt^2+u_3^2-d_iu_1^2=0\\
H_3:& 2nt^2+d_iu_1^2-d_iu_2^2=0
\end{cases}\]

From equations (\ref{eq:cCgamma-ci}) and (\ref{eq:cCgamma}) we have global points $Q_i$ on $H_i$ as follows
\[ Q_1=(b_i,c_i,d_ia_i) , \quad Q_2=(\barb_i, d_i\bara_i, \barc_i), \quad Q_3=(0,1,1)  \]
Then the tangent planes at these points are
\begin{eqnarray*}
L_1:&&   d_i'b_it-c_iu_2+a_iu_3  \\
L_2:&&  d_i'\barb_it-\bara_iu_3+\barc_iu_1  \\
L_3:&&  u_1- u_2
\end{eqnarray*}

Similarly the pairings are trivial outside of $p|n$, hence we only have to consider for those $p|n$. We can chose local solutions $P_p\in D_{\Lambda_i'}(\bQ_p)$ as following: For $p|d_i'$:
\[\begin{array}{ccc}
t=0,  &u_2=-u_1=1,& u_3^2=d_i
\end{array}\] such that $p\big| (c_i+a_iu_3), p\big| (\barc_i-\bara_iu_3)$,  as  $(c_i-a_iu_3)(c_i+a_iu_3)=d_i'b_i^2$, so we can chose $u_3$ such that $p\big| (c_i+a_iu_3)$, thus we also have $p\big| (\barc_i-\bara_iu_3)$ follows from condition (\ref{eq-cond:cCgamma}). For $p|d_i$,
\[\begin{array}{ccc}
 t=1, & u_3=0, u_1^2=-d_i', &u_2=-i_p u_1
\end{array}\] with $ p\big|(d_i'\barb_i-\barc_iu_1)$, since we have $(d_i'\barb_i-\barc_iu_1)(d_i'\barb_i+\barc_iu_1)=n\bara_i^2$, note here we have used $i_p\in \bQ_p$ such that $i_p^2=-1$ as in Theorem \ref{lem-LiTian}.\\

Now we begin to compute Cassels pairings:

First we compute $\Babig{\Lambda_i'}{\Lambda_j'}$  with $i\not=j$, then by Cassels pairing: $$\Babig{\Lambda_i'}{\Lambda_j'}=\Babig{(d_i,d_i,1)}{(d_j,d_j,1)}=\prod_p\brBig{L_1L_2(P_p), d_j}_p$$
Note that $\Leg{d_j}p=1$ if $p \nmid   d_j$, so this pairing is also trivial at $p|d_j'$, thus we only need to consider at those $p|d_j$, and for $p|d_j$:
$$\brBig{L_1 L_2(P_p), d_j}_p=\brBig{(-c_i+a_iu_3)(-\bara_i u_3-\barc_i), d_j}_p=\brBig{c_i\barc_i, d_j}_p=\Leg{c_i\barc_i}p$$
 Thus we have:
\[\Babig{(d_i,d_i,1)}{(d_j,d_j,1)}=\Leg{c_i\barc_i}{d_j} \]\quad\\

Then we compute $\Babig{\Lambda_i'}{\Lambda_j}$  with $i\not=j$:  then from Cassels pairing we have $$\Babig{\Lambda_i'}{\Lambda_j}=\Babig{(d_i,d_i,1)}{(1,d_j,d_j)}=\prod_p\brBig{L_2 L_3(P_p), d_j}_p$$
Since $\Leg{d_j}p=1$ if $p \nmid   d_j$, so we only need to consider at those $p|d_j$, and for $p|d_j$:
$$\brBig{L_2 L_3(P_p), d_j}_p=\brBig{-2(-\bara_i u_3-\barc_i), d_j}_p=\brBig{\barc_i, d_j}_p=\Leg{\barc_i}p$$
Whence
\[\Babig{(d_i,d_i,1)}{(1,d_j,d_j)}=\Leg{\barc_i}{d_j} \]\quad\\

Finally we compute the pairing $\Babig{\Lambda_i'}{\Lambda_i}$, then by Cassels pairing $$\Babig{\Lambda_i'}{\Lambda_i}=\Babig{(d_i,d_i,1)}{(1,d_i,d_i)}=\prod_p\brBig{L_2 L_3(P_p), d_i}_p$$
From $\Leg{d_i}p=1$ if $p \nmid   d_i$, we know this pairing is trivial at $p|2d_i'\infty$, thus we suffice to compute at those $p|d_i$, and for $p|d_i$:
$$\brBig{L_2L_3(P_p),d_i}_p=\brBig{(1+i_p)u_1(d_i'\barb_i+\barc_iu_1), d_i}_p=\brBig{\barc_i(1+i_p), d_i}_p=(-1)^{\delta_p}\Leg{\barc_i}p$$
Cnsequently
\[\Babig{(d_i,d_i,1)}{(1,d_i,d_i)}=(-1)^{\delta_{d_i}}\Leg{\barc_i}{d_i} \]

This completes the proof.
\end{proof}\quad

To give the matrix representation of Cassels pairing under the base $\Lambda_1,\cdots, \Lambda_k, \Lambda_1',\cdots, \Lambda_k'$, we  introduce two matrices over $\bF_2$:
\[  \Psi=(\psi_{ij}), \quad D^*=\diag(1-h_8(d_1),\cdots,1-h_8(d_k))\]
where $\psi_{ii}=\ALeg{c_i}{d_i}$ and $\psi_{ij}=\ALeg{\gamma_j}{d_i}$ if $i\not=j$. By  Theorem \ref{lem-LiTian}, we have $1-h_8(d_i)\equiv \delta_{d_i} \pmod 2$. According to the definition of $A^*$ is \S3.3 and Proposition \ref{prop-anySha},  we know  the Cassels pairing under the base $\Lambda_1,\cdots, \Lambda_k, \Lambda_1',\cdots, \Lambda_k'$ has matrix representation of the form
\begin{equation}\label{eq:matrixrep}
\Ma{{A^*}+{\Psi}}{ {A^*}^T+D^*} {A^*+D^*} {A^*+{A^*}^T }
\end{equation}
where we used that Cassels pairing is symmetric to derive $A^*+\Psi$ is symmetric. Then we have:

\begin{thm} \label{mainthm2}
Let $n=d_1\cdots d_k$ be a square-free positive integer with all prime factors congruent to $1 \pmod 8$, satisfying the following ideal class group conditions:
\begin{itemize}
\item $h_4(d_i)=1,  1\le i\le k$;
\item $\Leg{p_{i_1}}{p_{i_2}}=1$ for any prime divisor $p_{i_j}$ of $d_{i_j}$ with $i_1\not=i_2$;
\end{itemize}
If  $A^*$ of $n$ satisfying
\begin{enumerate}
\item[(1)] $A^*$ is symmetric;
\item[(2)] $A^*+D^*$  is non-singular;
\end{enumerate}
Then:
\[\rank_{\bZ} \En(\bQ)=0, \;\;\Sha(\En/\bQ)[2^\infty]\simeq \brbig{\bZ/2\bZ}^{2k}\]
\end{thm}


\begin{proof} From $A^*$ symmetric  we know $A^*+{A^*}^T=0$, then the Cassels pairing has matrix representation\[\Ma{*}{{A^*}^T+D^*}{{A^*}+D^*}{}\] by (\ref{eq:matrixrep}). Since ${A^*}+D^*$ is non-singular, then the Cassels pairing on $\Sel_2(\En)/\En(\bQ)[2]$ is non-degenerate, whence as explained in the introduction
\[\rank_{\bZ} \En(\bQ)=0, \;\;\Sha(\En/\bQ)[2^\infty]\simeq \brbig{\bZ/2\bZ}^{2k}\]
\end{proof}

Now we can prove Theorem \ref{mainthm}:
\begin{proof} From $h_8(n)=k$ or $h_8(n)=k-1$ with $[(2,\sqrt{-n})]\not\in 4\cA$ we know $A^*=0$ by the epimorphism (\ref{eq:R*-8rank}) and Proposition \ref{gauss-Redei-c-version}, hence $A^*$ is symmetric. From $h_8(d_i)=0$ we get $D^*=\diag\brbig{1,\cdots,1}$ and
$A^*+D^*$ is non-singular, whence (1) and (2) of Theorem \ref{mainthm2} are satisfied in this case. Then from Theorem \ref{mainthm2} we have
\[\rank_{\bZ} \En(\bQ)=0, \;\;\Sha(\En/\bQ)[2^\infty]\simeq \brbig{\bZ/2\bZ}^{2k}\]
The proof is completed.
\end{proof}


For the case $k=2$ of Theorem \ref{mainthm2} can be characterized by:
\begin{cor}\label{cor2} Let $n=d_1d_2$ be a square-free positive integer with all prime factors congruent to $1 \pmod 8$ with $h_4(d_1)=h_4(d_2)=1$ such that: $\Leg{p_{i_1}}{p_{i_2}}=1$ for any prime divisor $p_{i_j}$ of $d_{i_j}$ with $i_1\not=i_2$. Then the   conditions  (1), (2) on $A^*$ in Theorem \ref{mainthm2} are equivalent to:
\begin{itemize}
\item $\Leg{d_1}{d_2}_4=\Leg{d_2}{d_1}_4$;
\item Either exact one of $h_8(d_1), h_8(d_2)$ is $0$ and  $\Leg{d_1}{d_2}_4$ is $-1$, or both $h_8(d_1), h_8(d_2)$ equal to $0$.
\end{itemize}
\end{cor}

\begin{proof} By the definition of $A^*$ with $k=2$,  we may chose $c_1=c_2$ such that
\begin{equation}\label{eq:c_1-2=}
c_1^2=d_1a_1^2+d_2b_1^2
\end{equation}
then $A^*$ is symmetric if and only if $\Leg{c_1}{d_1}=\Leg{c_1}{d_2}$. Since
\[\Leg{c_1}{d_1}=\Leg{c_1^2}{d_1}_4=\Leg{d_1a_1^2+d_2b_1^2}{d_1}_4=\Leg{d_2}{d_1}_4\cdot \Leg{ b_1}{d_1}=\Leg{d_2}{d_1}_4\]
where $\Leg{b_1}{d_1}=1$ follows from equation (\ref{eq:c_1-2=}) modulo  $b_1$. Similarly we get $\Leg{c_1}{d_2}=\Leg{d_1}{d_2}_4$, whence (1) is equivalent to $\Leg{d_1}{d_2}_4=\Leg{d_2}{d_1}_4$. Then under this condition $A^*$ has the form\[ \ALeg{c_1}{d_1}\cdot \left(
         \begin{array}{cc}
           1 & 1 \\
           1 & 1 \\
         \end{array}
       \right)
\]
If $h_8(d_1)=h_8(d_2)=0$ then $D^*=\diag(1,1)$ and  $A^*+D^*$ is non-singular. If $h_8(d_1)=h_8(d_2)=1$, then $D^*=0$ and $A^*+D^*$ is singular. If exact one of $h_8(d_1)$ and $h_8(d_2) $ is $0$, then $A^*+D^*$ is non-singular if and only if $\Leg{c_1}{d_1}=-1$. Hence this corollary is proved.
\end{proof}
\begin{rem}\label{rem:non-empty} For $k=2$, the condition of Theorem \ref{mainthm} is non-empty.  Since from Theorem \ref{mainthm}, we know the $8$-rank condition  is equivalent to $h_8(d_1)=h_8(d_2)=0$ and $[\fd_1]\in 4\cA$, where the latter is equivalent to $\Leg{d_1}{d_2}_4=\Leg{d_2}{d_1}_4=1$. By Theorem \ref{lem-LiTian} we know $h_8(d_i)=0$ is equivalent to $\ALeg2{d_i}_4\equiv \frac{d_i-9}8 \pmod 2$, where $\ALeg2{d_i}_4$ is defined similarly as additive Legendre symbol. Then everything reduces to find   $n$ satisfying certain residue symbol properties, then from independence of residue symbol property over $\bQ(i)$ in our coming paper \cite{zhangjiewang2015congDist}, we know  there are infinitely many such $n$, in particular we can get their density.
\end{rem}


\textbf{Acknowledgements}\\

I am greatly indebted to Professor Ye Tian, my supervisor, for many  instructions and suggestions! I would like to thank Li Cai and Zhe Zhang for carefully reading the paper and giving valuable comments. I also thank Professor Qin Yue for  instructions on ideal class group.

\bigskip

\noindent Zhangjie Wang, \\
 Yau  Mathematical Sciences Center, \\
Tsinghua University, Beijing 100084,China. \\
{\it {zjwang@math.tsinghua.edu.cn}  }

\end{CJK}
\end{document}